\documentclass{article}
\usepackage{amsmath,amssymb,amsthm}
\usepackage[all]{xy}

\newcommand\Fzrh{F_{(0,r)}}
\newcommand\Fzr{F_{(0,r]}}
\newcommand\Fr{F_{[r,r(K,\partial)]}}
\newcommand\Frr{F_{[r,r]}}

\newcommand\Fjrr{F_{j,[r,r]}}
\newcommand\Fjrj{F_{j,[r_j,r_j]}}

\newcommand\Frp{F_{[r,r(K',\partial)]}}

\newcommand\Frs{F_{sp,[r,r]}}

\newcommand\FrJs{F_{sp,[r_J,r_J]}}

\newcommand\Fp{F_+}

\newcommand\Fm{F_-}
\newcommand\Fpim{F_{\pi}}

\newcommand\Fpip{F^{\pi}}

\newcommand{\rK}{r(K,\partial)}
\newcommand{\rKJ}{r(K,\partial_J)}

\newcommand{\rKp}{r(K',\partial)}

\newcommand\V{\mathbf{V}}
\newcommand\MM{\mathbf{M}}
\newcommand\PPhi{\varphi}
\newcommand\PPsi{\psi}

\newcommand\KT{K\langle T\rangle}

\newcommand\KpT{K'\langle T\rangle}

\newcommand\KX{K[[X]]}
\newcommand\KXT{K[[X]]\langle T\rangle}

\newcommand\Tat{K\langle X/t\rangle}
\newcommand\Tas{K\langle X/s\rangle}
\newcommand\Tar{K\langle X/s\rangle}

\newcommand\KXrz{K[[X/r]]_0}
\newcommand\KXrzT{K[[X/r]]_0\langle T\rangle}

\newcommand\KXrKz{K[[X/r(K,\partial)]]_0}
\newcommand\KXrKzT{K[[X/r(K,\partial)]]_0\langle T\rangle}

\newcommand\KXrK{K\{X/r(K,\partial)\}}

\newcommand\KXr{K\{X/r\}}
\newcommand\KXrT{K\{X/r\}\langle T\rangle}

\newcommand\KpXr{K'\{X/r\}}
\newcommand\KpXrT{K'\{X/r\}\langle T\rangle}

\newcommand\KXt{K\{X/t\}}

\newcommand\KXrp{K\{X/r+\}}
\newcommand\KXrpT{K\{X/r+\}\langle T\rangle}

\newcommand\KTJ{K\langle T_J\rangle}
\newcommand\KTj{K\langle \mathbf{T}_j\rangle}

\newcommand\KXTJ{K[[X]]\langle T_{j,J}\rangle}

\newcommand\KXTT{K[[X]]\langle \mathbf{T}\rangle}
\newcommand\KXrTT{K\{X/r\}\langle \mathbf{T}\rangle}
\newcommand\KXrpTT{K\{X/r+\}\langle \mathbf{T}\rangle}

\newcommand\KXrTj{K\{X/r\}\langle \mathbf{T}\rangle}
\newcommand\KXrpTj{K\{X/r+\}\langle \mathbf{T}\rangle}
\newcommand\KXrpTJ{K\{X/r+\}\langle T_{j,J}\rangle}
\newcommand\KXrTJ{K\{X/r\}\langle T_{j,J}\rangle}

\newcommand\RT{R\langle T\rangle}
\newcommand\RTJ{R\langle T_J\rangle}

\newcommand{\supp}{\mathrm{supp}\ }

\newcommand\C{Mod(K\langle T\rangle)}
\newcommand\Cf{Mod^f(K\langle T\rangle)}

\newcommand\CC{Mod(K\langle T\rangle,Mod(K\langle T\rangle))}

\newcommand\CJ{Mod(K\langle T_J\rangle)}
\newcommand\CJf{Mod^f(K\langle T_J\rangle)}

\newcommand\Cj{Mod(K\langle \mathbf{T}_{j}\rangle)}
\newcommand\Cjf{Mod^f(K\langle \mathbf{T}_j\rangle)}

\newcommand\CKXrT{Mod(K\{X/r\}\langle T\rangle)}

\newcommand\CKXT{Mod(K[[X]]\langle T\rangle)}

\newcommand\CKXrzT{Mod(K[[X/r]]_0\langle T\rangle)}

\newcommand\CKXrKzT{Mod(K[[X/r(K,\partial)]]_0\langle T\rangle)}

\newcommand\CKXrpT{Mod(K\{X/r+\}\langle T\rangle)}

\newcommand\CRT{Mod(R\langle T\rangle)}
\newcommand\CRTJ{Mod(R\langle T_J\rangle)}

\newcommand\A{\mathcal{A}b}

\newcommand\fC{\mathcal{C}}
\newcommand\fD{\mathcal{D}}

\newcommand\D{\mathcal{D}(R,\partial)}
\newcommand\DJJ{\mathcal{D}(R,\partial_J)}

\renewcommand\H{\mathrm{Hom}}
\newcommand\Ex{\mathrm{Ext}}
\newcommand\E{\mathrm{End}}

\newcommand\rhozt{\tilde{\rho}_0}
\newcommand\rhot{\tilde{\rho}}
\newcommand\rhos{\rho^{\#}}

\newcommand\TTFAE{The following are equivalent.}
\newcommand\TFAE{the following are equivalent.}

\newcommand\R{\mathbb{R}}
\newcommand\Q{\mathbb{Q}}
\newcommand\Z{\mathbb{Z}}
\newcommand\N{\mathbb{N}}

\newcommand\f{\mathbf{f}}
\renewcommand\L{\mathbf{L}}

\newcommand\eps{\varepsilon}

\theoremstyle{plain}

\newtheorem{theorem}{Theorem}[section]
\newtheorem{proposition}[theorem]{Proposition}
\newtheorem{lemma}[theorem]{Lemma}
\newtheorem{corollary}[theorem]{Corollary}

\theoremstyle{definition}

\newtheorem{definition}[theorem]{Definition}

\begin{document}

\title{Duality for differential modules over \\ complete non-archimedean valuation fields of characteristic zero}
\author{Shun Ohkubo
\footnote{
Graduate School of Mathematics, Nagoya University, Furocho, Chikusaku, Nagoya 464-8602, Japan. E-mail address: shun.ohkubo@gmail.com}
}
\maketitle

\begin{abstract}
Let $K$ be a complete non-archimedean valuation field of characteristic $0$, with non-trivial valuation, equipped with (possibly multiple) commuting bounded derivations. We prove a decomposition theorem for finite differential modules over $K$, where decompositions regarding the extrinsic subsidiary $\partial$-generic radii of convergence in the sense of Kedlaya-Xiao. Our result is a refinement of a previous decomposition theorem due to Kedlaya and Xiao. As a key step in the proof, we prove a decomposition theorem in a stronger form in the case where $K$ is equipped with a single derivation. To achieve this goal, we construct an object $f_{0*}L_0$ representing the usual dual functor and study some filtrations of $f_{0*}L_0$, which is used to construct the direct summands appearing in our decomposition theorem.
\end{abstract}

\tableofcontents

\section{Introduction}

In this paper, we prove a decomposition theorem of finite differential modules over a complete non-archimedean valuation field $K$ of characteristic $0$, with non-trivial valuation, equipped with multiple commuting derivations, which are non-zero and bounded with respect to the given valuation. Precisely speaking, our decomposition is regarding the extrinsic subsidiary $\partial$-generic radii of convergence in the sense of \cite{kx}, which is an invariant generalizing the generic radii of convergence introduced by Dwork.

Let us state the main result of this paper. Let $K$ be as above, $||$ denote the valuation of $K$, $\partial_j$ for $j\in J$ with $J\neq\emptyset$ the derivations on $K$, which are commutative each other. Let $p(K)$ denote the characteristic of the residue field of $K$. We set $\omega(K)=|p(K)|^{1/(p(K)-1)}$ if $p(K)>0$ and $\omega(K)=1$ if $p(K)=0$. We assume that the derivations $\partial_j$ are bounded, i.e., the action of $\partial_j$ on $K$ has a finite operator norm. Let $r(K,\partial_j)$ denote the ratio $\omega(K)/|\partial_j|_{sp,K}$, where $|\partial_j|_{sp,K}$ denotes the spectral norm of the action of $\partial_j$ on $K$. A finite differential module over $K$ is a finite dimensional vector space over $K$ equipped with a family of commuting differential operators relative to $\partial_j$ for $j\in J$. For each $j\in J$, we can define the extrinsic subsidiary $\partial_j$-generic radii of convergence, by encoding the spectral norms of the differential operator relative to $\partial_j$, which is a set of $dim V$ real numbers $r$ with multiplicity taking values in $(0,r(K,\partial_j)]$. 
For each $J$-tuple $r_J=(r_j)\in \Pi_{j\in J}(0,r(K,\partial_j)]$, there exists a maximum subspace $V_{r_J}$ of $V$ such that for $j\in J$, $V_{r_J}$ is stable under the action of $\partial_j$ and the extrinsic subsidiary $\partial_j$-generic radii of convergence of $V_{r_J}$ consists only of $r_j$'s.

\begin{theorem}\label{main theorem}
Let $V$ be a finite differential module over $K$.

1. There exist only finitely many $r_J\in \Pi_{j\in J}(0,r(K,\partial_j)]$ such that $V_{r_J}\neq 0$. 

2. The obvious morphism $\oplus_{r_J\in \Pi_{j\in J}(0,r(K,\partial_j)]}V_{r_J}\to V$ is an isomorphism.
\end{theorem}

Besides the above result, we also prove a base change property and a rationality result in an appropriate sense.

The decomposition in Theorem \ref{main theorem} is a refinement of the one stated in \cite[Theorem 1.5.6]{kx}, where $\#J<+\infty$ is assumed and the direct summands $V_r$ are parametrized by $r\in (0,1]$. Hence our decomposition can be applied to differential modules, which cannot be decomposed by previous decomposition results. Even in the case $\#J=1$ and $p(K)>0$, Theorem \ref{main theorem} is known only in the case where $K$ is of rational type in the sense of \cite{kx} such as the completion of $\Q_p(X)$ with respect of some Gauss valuation equipped with derivation given by $d/dX$.

Decompositions as in Theorem \ref{main theorem} in the case of $p(K)>0$ are studied in a context of $p$-adic differential equations by Dwork-Robba, Christol-Dwork, Kedlaya-Xiao, and Poineau-Pulita and so on (\cite{dr,cd,kx,pp}). Beside applications to $p$-adic differential equations, variations of such decompositions plays a fundamental role in the study of differential Swan conductors due to Kedlaya, Xiao (\cite{ked,xiao}). Even in the case of $p(K)=0$, related decompositions appear in the study of Hukuhara-Levelt-Turrittin type decompositions in the sense of Kedlaya (\cite{good formalII}).

\subsection*{Technical aspects}

In the course of the proof of Theorem \ref{main theorem}, we introduce an essential idea on a duality of differential modules over $K$ in the case of $\#J=1$.

Let us start with noting the subtlety in the proof of Theorem \ref{main theorem} occurring only in the case of $\#J>1$. For a non-zero finite differential module over $K$, even if we have a nice decomposition $V=\oplus V_r$ regarding a single derivation $\partial_j$, it is not clear that the direct summands $V_r$ are stable under the action of the other $\partial_{j'}$'s as they are not $K$-linear.

To overcome this difficulty, we prove Theorem \ref{main theorem} in the case of $\#J=1$ in a stronger form so that we can prove Theorem \ref{main theorem}. In the rest of the introduction, we consider only in the case of $\#J=1$. We consider the category $\C$ of left modules over the ring $\KT$ of twisted polynomials over $K$ instead of the category of differential modules over $(K,\partial)$, which are isomorphic to each other. We have a contravariant endofunctor $D_0$ on $\C$ called a dual functor, which is constructed in an obvious way. The key ingredient in this paper is the object $f_{0*}L_0$ in the category $\CC$ of a left $\KT$-objects in $\C$, which ``represents'' $D_0$. The construction of $f_{0*}L_0$ can be done by using a variant $f_0$ of the morphism $f_{gen}^*$ introduced by Kedlaya-Xiao. We define a new dual functor $D$ by $\H(-,f_{0*}L_0)$. The underlying abelian group of $f_{0*}L_0$ is the direct product $K_+^{\N}$ of $\N$-copies of the underlying abelian group $K_+$ of $K$. Hence we can construct subobjects $M_{\lambda}$ of $f_{0*}L_0$ by considering some convergence conditions on the formal power series $\sum_{i=0}^{+\infty}a_iX^i$ associated to $(a_i)\in K_+^{\N}$. This part can be regarded as a generalization of the theory of Dwork-Robba. Then we study the differential modules of the form $\H(M,M_{\lambda})$ for $M\in \C$, where we exploit a techiniques developed by Dwork-Robba, Christol, and Poineau-Pulita (\cite{dr,chr,pp}). Then we regard $\H(M,M_{\lambda})$ and $\H(DM,M_{\lambda})$ as submodules of $\H(M,f_{0*}L_0)=DM$ and $\H(DM,f_{0*}L_0)=DDM$ respectively. As in the case of vector spaces, we can define submodules $F_-M$ and $F_+M$ of $M$ corresponding to the above two submodules respectively. By varying $M_{\lambda}$, we obtain submodules $F_{(0,r]}M$ and $F_{[r,\rK]}M$ of $M$ for $r\in (0,\rK]$. We define $F_{[r,r]}M:=F_{(0,r]}M\cap F_{[r,\rK]}M$ and prove the desired decomposition $M\cong\oplus F_{[r,r]}M$.

By using the results explained as above, we prove Theorem \ref{main theorem} for all $K$ satisfying the conditions in the beginning of the paper. However, we do not know, at this point, how to prove our rationality result by our method only. To prove our rationality result, we exploit a consequence of a theory of Christol-Dwork (\cite{cd}) after embedding $K$ into another $K'$ of rational type in the sense of \cite{kx}, where the theory of Christol-Dwork is applicavle.

\subsection*{Acknowledgement}

The author thanks Kiran Kedlaya for correspondence on the paper \cite{kx}. This work is supported by JSPS KAKENHI Grant-in-Aid for Scientific Research (C) 22K03227 and Grant-in-Aid for Young Scientists (B) 17K14161.

\subsection*{Notation and convention}

In this paper, a ring is an associative ring with unit. When $X$ is a ring, a field, or a left $R$-module for some ring $R$, unless otherwise is mentioned, we denote by $X_+$ the underlying abelian group of $X$ in an obvious sense. For an abelian group $A$ and a set $S$, we denote by $A^S$ (resp. $A^{(S)}$) the abelian group of families $(a_i)_i$ in $A$ indexed by $S$ (resp. such that $a_i=0$ for all but finitely many $i\in S$), where $(a_i)_i$ is denoted by $(a_i)$ for simplicity. When $X$ is an abelian group, a ring, a field, or a left $R$-module $M$, to denote an element $x$ of the underlying set of $M$, we write $x\in M$ for simplicity.

When $X$ is an abelian group, a ring, a field, or a left $R$-module for some ring $R$, $X_u$ denote the underlying set of $X$. Furthermore, an ultrametric function $||:X\to\mathbb{R}_{\ge 0}$ refers to a morphism $||:X_u\to \mathbb{R}_{\ge 0}$ of sets such that $|x-y|\le\max\{|x|,|y|\}$ and $|0|=0$ for $x,y\in X$. Assume $||^{-1}(\{0\})=\{0\}$ and $X\neq 0$. Let $f:X_u\to X_u$ be a map such that there exists $C\in\R$ such that $|f(x)|\le C|x|$ for $x\in X$. We define the operator norm of $f$ by $|f|_{op}=\sup\{|f(x)|/|x|;x\in X_u,x\neq 0\}$. Moreover we define the spectral norm of $f$ by $|f|_{sp}=\inf\{|f^i|^{1/i}_{op};i\in\N_{>0}\}$. When would like to specify $X$, we denote $|f|_{op},|f|_{sp}$ by $|f|_{op,X},|f|_{sp,X}$ respectively.

Although Theorem \ref{main theorem} concerns finite differential modules when the base field $K$ is equipped with (possibly) multiple derivations $\{\partial_j\}_{j\in J}$, except the last section, we restrict to the case $\#J=1$.

\section{Preliminaries}

In this section, we recall basic notions on categories and basic definitions and results on the category $\D$ of differential modules over differential rings $R$ equipped with a single derivation $\partial$.

\subsection{The category $Mod(A,Mod(B))$ and $Mod(A)$-valued bifunctor}\label{cat}

We recall basic definitions on categories and modules. See \cite{mit} for details.

In this paper, unless otherwise is mentioned, a category is assumed to be locally small. Let $\A$ denote the category of abelian groups.  Let $A$ be a ring and $\fC$ an abelian category. We define the category $Mod(A,\fC)$ of left $A$-objects in the category $\fC$. An object is a pair $(X,\rho)$, where $X\in \fC$ and $\rho:A\to \E_{\fC}(X)$ is a ring homomorphism. A morphism $(X',\rho')\to (X,\rho)$ is a morphism $\alpha:X'\to X$ such that $\alpha\circ (\rho'(a))=(\rho(a))\circ\alpha$ for $a\in A$. We denote an object $(X,\rho)$ of $Mod(A,\fC)$ by $X$ when no confusion arises. The category $Mod(A,\fC)$ is an abelian category. We define the category $Mod(A)$ of left $A$-modules by $Mod(A,\A)$, which coincides with the usual one. We will study the category of the form $Mod(A,Mod(B))$ with $B$ a ring. In terms of \cite[II,\S 1.14]{bourbaki}, an object of $Mod(A,Mod(B))$ is a (left) $((A,B),())$-multimodule. For an abelian category $\fD$ and a functor $F:\fD\to Mod(A)$, we denote by $F_+$ the functor given by the composition $()_+F$.

Let $\fC,\fD$ be abelian categories and $F:\fC\to \fD$ a covariant functor. We define the functor $Mod(A,F):Mod(A,\fC)\to Mod(A,\fD)$ associated to $F$, which is denoted by $F$ if no confusion arises, by defining $Mod(A,F)(X,\rho)=(FX,F\rho)$ for $(X,\rho)\in Mod(A,\fC)$, where we define $(F\rho)(a)=F(\rho(a))$ for $a\in A$, and defining $Mod(A,F)\alpha=F\alpha$ for $\alpha:(X',\rho')\to (X,\rho)$ a morphism in $Mod(A,\fC)$. Let $T:\fD\times\fC\to\A$ be an additive bifunctor which is contravariant in the first variable and covariant in the second one. Let $(X,\rho)\in Mod(A,\fC)$. Then $T(Y,X)$ for $Y\in\fD$ is regarded as a left $A$-module via the ring homomorphism $A\to\E(T(Y,X));a\mapsto T(Y,\rho(a))$. For morphisms $\alpha:X'\to X$ in $\fC$ and $\beta:Y'\to Y$ in $\fD$, $T(\beta,\alpha)$ defines a morphism $T(Y',X)\to T(Y,X')$ in $Mod(A)$. In this way, $T$ induces a bifunctor $\fD\times Mod(A,\fC)\to Mod(A)$, which is denoted by $T$ if no confusion arises. In this paper, we mainly consider the case where $\fC=\fD=Mod(B)$ for some ring $B$ and $T=\H$. In this case, $\H:Mod(B)\times Mod(A,Mod(B))\to Mod(A)$ is additive and left exact as the forgetful functor $Mod(A)\to \A$ is faithful.

For a ring homomorphism $f:R\to S$, we denote by $f^*:Mod(R)\to Mod(S),f_*:Mod(S)\to Mod(R)$ the pull-back and push-out functors respectively: for a left $S$-module $N=(N_+,\rho)$, $f_*N$ is the left $R$-module given by $(N_+,f_*\rho)$, where we define $f_*\rho=\rho\circ f$.  Let $\eta:f^*f_*\to id_{Mod(S)},\eps:id_{Mod(S)}\to f_*f^*$ denote the counit and unit for the adjoint functor $(f^*,f_*)$, which are natural transformations defined in an obvious way (\cite[II,5.2]{bourbaki}).

Finally we recall some terminology on natural transformations. 

Let $\fC$ be an abelian category. An endofunctor on $\fC$ is a (covariant or contravariant) functor whose domain and codomain are $\fC$. Let $F,G,H:\fC\to \fC$ be endofunctors on $\fC$, and $T:F\to G,S:G\to H$ natural transformations. We define the natural transformation $S\cdot T:F\to H$ by $(S\cdot T)_X=S_X\circ T_X$ for $X\in \fC$. We call $S\cdot T$ the vertical composition of $T$ followed by $S$.

Let $\fC$ be an abelian category. For $F,G:\fC\to \fC$ endofunctors on $\fC$, and $T:F\to G,S:F\to G$ natural transformations, we define the natural transformation $S+T:F\to G$ by $(S+T)_X=S_X+T_X$ for $X\in \fC$. Let $\fC$ be an abelian category, $F:\fC\to \fC$ an endofunctor, which is assumed to be representable. By Yoneda's lemma, the class of natural transformations $\E(F)$ forms a ring where the addition and multiplication given by $+,\cdot$ respectively defined as above.


\subsection{A base change lemma}

\begin{definition}
Let $A,B,C,D$ be rings, $f:C\to D,g:A\to B,i:A\to C,j:B\to D$ ring homomorphism such that $j\circ g=f\circ i$. Let $\eta:id_{Mod(B)}\to j_*j^*$ be the unit for the adjoint functor $(j^*,j_*)$, and $\epsilon:i^*i_*\to id_{Mod(C)}$ the counit for the adjoint functor $(i^*,i_*)$, and $\varphi:j^*g^*\to f^*i^*$ the natural isomorphism defined by the vertical composition $j^*g^*\to (j\circ g)^*\to (f\circ i)^*\to f^*i^*$, where each arrow is an obvious one. We define the natural transformation $\delta:g^*i_*\to j_*f^*$ as the vertical composition $(j_*f^*\epsilon)\cdot (j_*\varphi i_*)\cdot (\eta g^*i_*):g^*i_*\to j_*j^*g^*i_*\to j_*f^*i^*i_*\to j_*f^*$. For $M\in Mod(C)$, we have $\delta_M:g^*i_*M\to j_*f^*M$ coincides with the unique morphism satisfying, for $x\in M$, $\delta_M(1\otimes x)=1\otimes x$. We consider left $B$-modules $B\otimes C,D$, where we regard $B$ as a $(B,A)$-bimodule via $id_B,g$ respectively, $C$ as a left $A$-module via $i$, and $D$ as a left $B$-module via $j$. We define the morphism $\delta_0:B\otimes C\to D$ in $Mod(B)$ as the unique morphism satisfying, for $b\in B,c\in C$, $\delta_0(b\otimes c)=j(b)f(c)$.
\end{definition}

\begin{lemma}\label{base change}
Let notation be as above. Assume that $\delta_0$ is an isomorphism. Then $\delta:g^*i_*\to j_*f^*$ is a natural isomorphism.
\end{lemma}

\begin{proof}
We prove $\delta_C$ is an isomorphism. Let $\alpha:D\to j_*f^*C$ be the isomorphism defined by $\alpha(d)=d\otimes 1$. Since we have $g^*i_*C=B\otimes C$ by definition and $\delta_C=\alpha\circ\delta_0$, $\delta_C$ is an isomorphism.

Let $M\in Mod(C)$ be arbitrary. Let $\dots\to P_1\to P_0\to M$ be a free resolution of $M$. Since $g^*i_*$ and $j_*f^*$ commute with direct sum, $\delta_{P_1},\delta_{P_2}$ are isomorphisms. By the right exactness of $g^*i_*$ and $j_*f^*$, $\delta_M$ is an isomorphism.
\end{proof}


\subsection{The categories $\D$ and $\CRT$}\label{def}

We recall basic definitions and results on differential modules over differential rings in a point of view of categories. See \cite{pde,kx} for detail.

A differential ring is a commutative ring $R$ equipped with a derivation, i.e., a morphism $\partial:R_+\to R_+$ in $\A$ satisfying $\partial(r'\cdot r)=\partial(r')\cdot r+r'\cdot\partial(r)$ for $r',r\in R$. The category of differential modules over $R$, which is denoted by $\D$ in this paper, is as follows. An object of $\D$ is an $R$-module $V$ equipped a morphism $\partial:V_+\to V_+$ in $\A$ satisfying $\partial (r\cdot v)=\partial(r)\cdot v+r\cdot\partial(v)$ for $r\in R,v\in V$; a morphism of $\D$ is a morphism $\alpha:V'\to V$ in $Mod(R)$ such that $\alpha(\partial(v'))=\partial(\alpha(v'))$ for $v'\in V$. The category $\D$ is an abelian category, where the addition $\H(V',V)\times \H(V',V)\to \H(V',V)$ is given by $(\alpha_1+\alpha_2)(v')=\alpha_1(v')+\alpha_2(v')$ for $\alpha_1,\alpha_2\in \H(V',V),v'\in V'$. We define the forgetful functor $()_+:\D\to \A$ by, for $V\in \D$, defining $V_+$ as the underlying abelian group of $V$ (we forget the derivation), and, for a morphism $\alpha:V'\to V$, defining $(\alpha)_+$ as $\alpha$ itself. The forgetful functor $()_+$ is obviously faithful and exact.

\begin{definition}
We define the object $G(R,\partial)\in \D$, denoted by $G$ if no confusion arises, as the left $R$-module given by $(R_+^{(\N)},\xi:R\to \E_{\A}(R_+^{(\N)});c\mapsto ((q_i)_i\mapsto (cq_i)_i))$ equipped with the differential operator given by $\partial((q_i)_i)=(\partial(q_i)+q_{i-1})_i$, where we set $q_{-1}=0$. Let $\H(G(R,\partial),-):\D\to \A$ be the covariant morphism functor given by $G(R,\partial)$. Let $e=(1,0,0,\dots)\in G(R,\partial)_+$. We define the natural transformations $\PPhi:\H(G(R,\partial),-)\to (-)_+,\PPsi:(-)_+\to \H(G(R,\partial),-)$. We define $\PPhi$ as the natural transformation corresponding to $e\in G(R,\partial)_+$ in the sense of Yoneda's lemma. Precisely speaking, $\PPhi_V$ for $V\in\D$ coincides with $\PPhi_V:\H(G(R,\partial),V)\to V_+;s\mapsto s(e)$. We define $\PPsi_V$ for $V\in\D$ by $\PPsi_V(x)((q_i))=\sum_{i=0}^{+\infty}q_i\partial^i(x)$ for $x\in V,(q_i)\in G(R,\partial)$. We can easily verify that $\PPhi$ and $\PPsi$ are inverse each other. Consequently, $G(R,\partial)$ is a projective generator for $\D$ (\cite[Proposition 15.3]{mit}).
\end{definition}


\begin{definition}
We have the isomorphisms $\PPsi_G:G_+\to \H(G,G),\PPhi_G:\H(G,G)\to G_+$ with inverse each other. Let $\cdot:G_+\times G_+\to G_+$ denote the bilinear map given by $x\cdot y=\PPhi_G(\PPsi_G(y)\circ\PPsi_G(x))$. Then $\PPsi_G,\PPhi_G$ define ring isomorphisms $(G_+,\cdot)\to \H(G,G)^{op},\H(G,G)^{op}\to (G_+,\cdot)$ respectively inverse each other. By a straightforward calculation, we have $(q'_i)_i\cdot (q_i)_i=(\sum_{j=0}^i\sum_{h\ge j}q'_h\binom{h}{j}\partial^{h-j}(q_{i-j}))_i$ for $(q'_i)_i,(q_i)_i\in G_+$. Hence the ring $(G_+,\cdot)$ is nothing but the ring of twisted polynomials associated to the differential ring $R$ in the sense of \cite[Definition 5.5.1]{pde}. Hence we call the ring $(G_+,\cdot)$ the ring of twisted polynomials associated to $(R,\partial)$, which is denoted by $\RT$. Thus we denote $(G_+,\cdot)$ by $\RT$. We use the twisted polynomial notation: we set $T=(0,1,0,\dots,)\in \RT$. We define the ring homomorphism $i:R\to \RT$ by $i(q)=(q,0,\dots,)$ for $q\in R$. Then, for $j\in \N$, we have $i(q)T^j=(0,\dots,0,q,0,\dots)$, where $q$ sits at the $j$-th entry, and we can express $(q_i)\in \RT$ as $x=\sum_{i=0}^ni(q_i)T^i$, where $n\in\N$ satisfies the condition that $q_j=0$ for $j>n$. Consequently, $\RT$ is generated by the subset $i(R)\cup\{T\}$ as a ring.
\end{definition}


\begin{definition}
Let $M\in \CRT$. Assume that $i_*M$ is finite free. For simplicity, $e_1,\dots,e_m\in M$ is a basis of $M$ if $e_1,\dots,e_m$ is a basis of $i_*M$. 
An arbitrary $x\in M$ is uniquely expressed as $\sum_{i=0}^mi(c_j)\cdot e_j$ with $c_1,\dots,c_m\in R$. For $k\in\N$, we define the matrix $G_k=(g_{k,ij})\in M_m(R)$ by $T^k\cdot e_k=\sum_{h=1}^mi(g_{k,ij})\cdot e_h$ for all $k$. 
We repeat a similar construction for $M'\in \CRT$ such that $i_*M'$ is finite free with a basis $e'_1,\dots,e'_n\in M'$. Let $\alpha:M'\to M$ be a morphism. We define the matrix $X\in M_{mn}(R)$ by $\alpha(e_k)=\sum_{h=1}^ni(x_{kh})\cdot e'_h$ for all $k$. Then we have $XG'_1+\partial(X)=G_1X$.
\end{definition}


\begin{definition}
Since the forgetful functor $()_+:\D\to\A$ is representable by $G$, the class $\E(()_+)$ of natural transformations $()_+\to ()_+$ is a ring by Yoneda's lemma. We define the ring homomorphism $\rho:\RT\to \E(()_+)$ by the vertical composition $\rho(x)=\PPhi\cdot\H(\PPsi_G(x),-)\cdot\PPsi:()_+\to ()_+$ for $x\in \RT$. For $V\in\D$, we define the ring homomorphism $\rho_V:\RT\to \E_{\A}(V_+)$ by $\rho_V(x)=\rho(x)_V$ for $x\in\RT$. At this point, we can prove Lemma \ref{universality} below, which asserts the ring $\RT$ satisfies a universal property.

We define the functor $\MM:\D\to \CRT$ by $\MM V=(V_+,\rho_V)$ for $V\in \D$ and $M\alpha=\alpha$ for a morphism $\alpha:V'\to V$. We define the functor $\V:\C\to \D$ by $\V M=i_*M\in Mod(R)$ for $M\in Mod(\RT)$ equipped with the differential operator given by $\rho(T)\in\E_{\A}((i_*M)_+)=\E_{\A}(M_+)$ and $\V\alpha=i_*\alpha$ for a morphism $\alpha:M'\to M$. We have $\V\circ \MM=id_{\D}$ and $\MM\circ \V=id_{\C}$ by using Lemma \ref{universality}. Let $V\in \D$ and $\xi:R\to\E_{\A}(V_+)$ denote the left $R$-module structure of $V$. Then, in $\E_{\A}((\MM V)_+)=\E_{\A}(V_+)$, we have $\xi(c)=\rho_V(i(c))$ for $c\in R$ and $\rho(T)\in \E_{\A}((\MM V)_+)=\E_{\A}(V_+)$ coincides with the differential operator of $V$.
\end{definition}


\begin{lemma}\label{universality}
Let $R$ be a differential ring. We consider the data $(U,\mu,u)$ where $U$ is a ring, $\mu:R\to U$ is a ring homomorphism, $u\in U$ such that $u\cdot\mu(r)=\mu(r)\cdot u+\mu(\partial(r))$ for $r\in R$. Then there exists a unique ring homomorphism $f:\RT\to U$ such that $\mu=f\circ i$ and $f(T)=u$. Moreover, we have $f((q_i))=\sum_i\mu(q_i)u^i$ for $(q_i)\in \RT$.

The ring homomorphism $f:\RT\to U$ is called the ring homomorphism corresponding to the data $(U,\mu,u)$.
\end{lemma}

\begin{proof}
The uniqueness follows from the fact that $\RT$ is generated by the subset $i(R)\cup\{T\}$ as a ring. We define the object $MU'\in\D$ by $\mu_*U$ equipped with the differential operator given by the left multiplication by $u\in U$. Let $\rho:\RT\to\E_{\A}(U_+)$ denote the structure morphism of $MU'$. We define the ring homomorphism $\tau:U\to \E_{\A}(U_+)$ by $\tau(x)(y)=x\cdot y$ for $x,y\in U_+$. Then $\tau$ is an injection by $\tau(x)(1)=x$ for $x\in U_+$. We have $\tau\circ \mu=\rho\circ i,\rho(T)=\tau(u)$ by the definition of the functor $M$. Since $\RT$ is generated by the subset $i(R)\cup\{T\}$ as a ring, we have $\rho(\RT)\subset \tau(U)$. Hence there exists a unique ring homomorphism $f$ such that $\rho=\tau\circ f$. We can easily see that $f$ satisfies the desired condition.
\end{proof}


\begin{lemma}\label{cyclic}
Let notation be as above. Assume $R$ is a field.

\begin{enumerate}
\item The ring $\RT$ is an integral domain.
\item The ring admits a left (resp. right) division theorem. Consequently, any left (resp. right) ideal of $\RT$ is principal. For any non-zero left (resp. right) ideal of $\RT$, there exists uniquely a generator $P\in \RT$ of the form $P=T^i+\sum_{j<i}q_jT^j$ with $i\in\N$.

\item Let $P\in \RT$ be non-zero. We write $P=q_iT^i+\sum_{j<i}q_jT^j$ with $i\in\N$ and $q_i\in R^{\times}$. Then $dim i_*(\RT/\RT\cdot P)=i$. 

\item Let $M\in\CRT$. Then $M$ is of finite length if and only if $i_*M$ is of finite dimension.

\item Let $M\in\CRT$ be of finite length. Then there exists an exact sequence of the form $0\to \RT\to \RT\to M\to 0$. As a consequence of part 1, the second morphism is given by the multiplication by some element $P\in \RT$ by right.
\end{enumerate}






\end{lemma}

\begin{proof}
See \cite[Proposition 5.2, Theorem 5.3,Remark 5.5]{pde} for parts 1,2,5. Part 3 is proved by using the left division theorem  To see part 4, note that $\RT$ is not irreducible as a left $\RT$-module by $0\subsetneq\RT\cdot T\subsetneq\RT$.
\end{proof}

\begin{definition}
\begin{enumerate}
\item We regard the $R$-module $R$ as a left $R$-object in $Mod(R)$ via the ring homomomorphism $R\to\E(R_+);x\mapsto (y\mapsto x\cdot y)$ given by the multiplication on $R$. Let $M\in \CRT$. Let $\xi:R\to\E_{\A}(\H(i_*M,R))$ be the ring homomoprhism defined by the left $R$-structure on $R$. Let $u\in \E_{\A}(\H(i_*M,R))$ be the endomorphism defined by $u(\chi)(m)=\partial(\chi(m))-\chi(T\cdot m)$ for $\chi\in \H(i_*M,R)$ and $m\in M$. Then the data $(\E(\H(i_*M,R)),\xi,u)$ satisfies a compatibility condition as in \ref{universality}. We define $D_0M$ as the left $\RT$-module given by $\H(i_*M,R)\in\A$ equipped with the ring homomorphism $\RT\to\E_{\A}(\H(i_*M,R))$ corresponding to the above data obtained in this way. We can see that for a morphism $\alpha:M'\to M$, the morphisn $\H(i_*\alpha,R)$ defines a morphism $D_0M\to D_0M'$. By definition, $(D_0M)_+=\H(i_*M,R)$ and we have $(i(r)\cdot\chi)(m)=r\cdot \chi(m)$ and $(T\cdot \chi)(m)=\partial(\chi(m))-\chi(T\cdot m)$ for $q\in R$ and $\chi\in \H(i_*M,R)$ and $m\in M$.
\item For $M\in \CRT$, we define the morphism $c_{0,M}:M\to D_0D_0M$ by $c_{0,M}(x)(\chi)=\chi(x)$ for $\chi\in D_0M$ and $x\in M$.
\end{enumerate}
\end{definition}

\begin{definition}
We define the object $L(R,\partial)\in \CRT$ as $\MM R$, where $R$ is regarded as an object of $\D$ in an obvious way. Note that we have an exact sequence $0\to \RT\to \RT\to L(R,\partial)\to 0$, where the second morphism is given by the right multiplication by $T$, the third one is given by the unique morphism sending $1\in \RT$ to $1\in L(R,\partial)_+=R_+$.
\end{definition}


\subsection{More constructions}

Let $K$ be a complete non-archimedean valuation field of characteristic $0$ with non-trivial valuation equipped with a bounded non-zero derivation. We define $p(K),\omega(K)$ as in the introduction.

We say that $M\in\C$ is of finite dimension if $i_*M$ is of finite dimension $Mod(K)$. Let $\Cf$ denote the full subcategory of $\C$ consisting of objects of finite dimension. The category $\Cf$ is an abelian subcategory of $\C$


\begin{definition}
For $M\in \Cf$, we define the map $m(M):(0,r(K,\partial)]\to\N$. When $M=0$, we define $m(M)(r)=0$ for all $r$. When $M$ is irreducible, we choose an ultrametric function $|\ |:M\to \R_{\ge 0}$ which gives a norm on $i_*M$. Then we define $m(M)(r)=\dim M$ if $r=\omega(K)/|T\cdot|_{sp}$, where $T\cdot:M_+\to M_+$ denotes the endomorphism on $M_+$ given by the left multiplication by $T$, 
and $m(M)(r)=0$ otherwise. When $M$ is arbitrary, let $\{M_j\}_{j\in S}$ be a finite family of objects in $\Cf$ such that the isomorphic class of $\{M_j\}_{j\in S}$ gives a Jordan-Holder factors of $M$ with multiplicity. We define $m(M)(r)=\sum_{j\in S}m(M_j)(r)$ for all $r$. Note that $m(M)(r)$ coincides with the multiplicity of $r$ (resp. $r/r(K,\partial)$) in the extrinsic (resp. intrinsic) subsidiary  generic $\partial$-radii of convergence (\cite[Definition 1.2.8]{kx}). Hence $m(M)$ is a function possessing the same information as the extrinsic subsidiary generic $\partial$-radii of convergence. It is known that the function $m(M)$ is independent of the choices of the $\{M_j\}$ and the ultrametric function. We define the support of $m(M)$ as $\supp m(M)=\{r\in (0,r(K,\partial)];m(M)(r)\neq 0\}$. 
\end{definition}



\begin{lemma}\label{multiplicity property}
Let $M\in \Cf$.

1. The support $\supp m(M)$ is a finite set and $\# \supp m(M)\le \dim M$.

2. The function $M\mapsto m(M)$ is additive on $\Cf$, where the addition of functions $m_1,m_2:[0,\rK)\to\N_{\ge 0}$ are given by $(m_1+m_2)(r)=m_1(r)+m_2(r)$.

3. We have $\sum_{r\in (0,r(K,\partial)]}m(M)(r)=\dim M$.

4. We have $m(D_0M)=m(M)$.
\end{lemma}
\begin{proof}
Parts 1,3 are obvious. Parts 1,4 are \cite[Lemma 1.2.9 (a),(b)]{kx}
\end{proof}


\begin{definition}
Let $Y$ be a subset of $(0,r(K,\partial)]$. Let $M\in \Cf$. We consider the subset $S(M;Y)$ of submodules of $M$ consisting of $M'$ such that $\supp m(M')\subset Y$, which is equipped with the partial order with respect to the inclusion relation. For $M',M''\in S(M;Y)$, the submodule $M'+M''$ of $M$ generated by $M'$ and $M''$ belongs to $S(M;Y)$ as $M'+M''\cong (M'\oplus M'')/M'\cap M''$ and the additivity of $m(-)$. Since $M$ is Artinian, there exists a maximal element $N$ in $S(M;Y)$. For $M'\in S(M;Y)$, we have $M'\subset N$ by $N\subset M'+N$ and $M'+N\in S(M;Y)$. Hence $N$ is the maximum element in $S(M;Y)$, which is denoted by $F_{sp,Y}M$. We have the canonical injection $F_{sp,Y}M\to M$ by definition, which is denoted by $I(Y)_M$.

Let $\alpha:M'\to M$ be a morphism in $\Cf$. By $\alpha(F_{sp,Y}M')\subset F_{sp,Y}M$, we can define a morphism $F_{sp,Y}\alpha:F_{sp,Y}M'\to F_{sp,Y}M$ in an obvious way. The function $F_{sp,Y}$ defines an endofunctor on $\Cf$ and $I(Y)$ defines a natural transformation $F_{sp,Y}\to id_{\Cf}$. 
\end{definition}

Let $X=\{X_{\lambda}\}_{\lambda\in\Lambda}$ be a partition of $(0,r(K,\partial)]$. For $M\in\Cf$, there exists finitely many $\lambda\in\Lambda$ such that $F_{sp,X_{\lambda}}M\neq 0$ by Lemma \ref{multiplicity property}. Hence we can define the direct sum endofunctor $\oplus_{\lambda\in\Lambda}F_{sp,X_{\lambda}}$ on $\Cf$ and the natural transformation $I(X):\oplus_{\lambda\in\Lambda}F_{sp,X_{\lambda}}\to id_{\Cf}$ defined by the $I(X_{\lambda})$'s in an obvious way.

\begin{lemma}\label{multiplicity lemma}
Let $X=\{X_{\lambda}\}_{\lambda\in\Lambda}$ be as above.
\begin{enumerate}
\item The natural transformation $I(X)$ is a pointwise monomorphism, that is, $I(X)_M$ is a monomorphism for $M\in \Cf$.
\item For $M\in \Cf$, \TFAE
\begin{enumerate}
\item The morphism $I(X)_M$ is an isomorphism.

\item For $\lambda\in \Lambda$, $\dim F_{sp,X_{\lambda}}M\ge \sum_{r\in X_{\lambda}}m(M)(r)$.
\end{enumerate}
\end{enumerate}
\end{lemma}

\begin{proof}

1. Let $M\in \Cf$. Let $\lambda_1,\dots,\lambda_n$ be distinct elements of $\Lambda$. Then $F_{sp,X_{\lambda_1}}M\cap (F_{sp,X_{\lambda_2}}M+\dots+F_{sp,X_{\lambda_n}}M)=0$ since $\supp m(F_{sp,X_{\lambda_1}}M\cap (F_{sp,X_{\lambda_2}}M+\dots+F_{sp,X_{\lambda_n}}M))\subset \supp m(F_{sp,X_{\lambda_1}}M) \cup \supp m(F_{sp,X_{\lambda_2}}M+\dots+F_{sp,X_{\lambda_n}}M)\subset X_{\lambda_1}\cap (X_{\lambda_2}\cup \dots\cup X_{\lambda_n})=\emptyset$.

2. Assume (a) holds. We have $m(M)=\sum_{\lambda\in \Lambda}m(F_{sp,X_{\lambda}}M)$. For $\lambda\in \Lambda$, we have $m(M)(r)=m(F_{sp,X_{\lambda}}M)(r)$ for $r\in X_{\lambda}$. Hence, $\dim F_{sp,X_{\lambda}}M=\sum_{r\in X_{\lambda}}m(F_{sp,X_{\lambda}}M)(r)=\sum_{r\in X_{\lambda}}m(M)(r)$.

Assume (b) holds. We have $\sum_{\lambda\in\Lambda}\dim F_{sp,X_{\lambda}}M\ge\sum_{\lambda\in\Lambda}\sum_{r\in X_{\lambda}}m(M)(r)=\sum_{r\in \cup_{\lambda\in\Lambda}X_{\lambda}}m(M)(r)=\sum_{r\in (0,r(K,\partial)]}m(M)(r)=\dim M$. Hence $I(X)_M$ is an isomorphism by part 1.
\end{proof}

We define the natural transformation $I_{sp}:\oplus_{r\in (0,r(K,\partial)]}\Frs\to id_{\Cf}$ by $I(X)$ for the family $X=\{[r,r]\}_{r\in (0,r(K,\partial)]}$. Note that, by Lemma \ref{multiplicity lemma}, $I_{sp}$ is a pointwise monomorphism. Also note that for $M\in\Cf$, $I_{sp,M}$ is an isomorphism if and only if $\dim \Frs M\ge m(M)(r)$ for all $r\in (0,r(K,\partial)]$.

\section{Duality}

In this section, we construct a contravariant endofunctor $D$ on $\C$, which is represented by an object $f_{0*}L_0$ of $\CC$ and is naturally isomorphic to the endofunctor $D_0$. We also construct a natural transformation $c:id_{\C}\to DD$. We give calculations on $D$ and $c$ used in the rest of the paper.

\noindent Convention. In the rest of the paper, except the last section \S \ref{decJ}, let $K$ be a complete non-archimedean valuation field of characteristic $0$ equipped with a derivation, where  the valuation $|\ |$ is non-trivial and the derivation is non-zero and bounded with respect to the valuations $|\ |$.



\subsection{Subrings of the ring of formal power series}

We recall the definition of various subrings defined in \cite{pde}.

We define the ring $\KX$ of formal power series over $K$ as the abelian group $K_+^{\N}$ equipped with the multiplication given by the convolution. 
We use the notation $(a_i)_i$, $(a_i)$ for simplicity, to denote an element of $\KX$ rather than the power series notation $\sum_ia_iX^i$. 
Let $r\in (0,+\infty)$. Let $\KXrz$ denote the subring of $\KX$ consisting of $(a_i)$ satisfying the condition $sup_i |a_i|r^i<+\infty$. Let $\KXr$ be the subring of $\KX$ defind by $\cap_{s\in (0,r)}K[[X/s]]_0$. Let $\KXrp$ be the subring of $\KX$ defind by $\cup_{s\in (r,+\infty)}K[[X/s]]_0$. We have $\KXrz\subset \KXr\subset \KXrp$ and, for $s\le r$, $\KXrz\subset K[[X/s]]_0,\KXr\subset K\{X/s\},\KXrp\subset K\{X/s+\}$. We define the $r$-Gauss valuation $||_r$ as the ultrametric function $|\ |_r:\KXrz\to\R_{\ge 0}$ given by $|(a_i)|_r=sup_i|a_i|r^i$; the multiplicativity can be seen by $|(a_i)\cdot (b_i)|_r=lim_{s\to r-0}|(a_i)\cdot (b_i)|_s=lim_{s\to r-0}(|(a_i)|_s|(b_i)|_s)=lim_{s\to r-0}|(a_i)|_s|lim_{s\to r-0}|(b_i)|_s$, 
where we regard $\KXrz$ as a subring of the subring $\Tas$ of $\KX$ consisting of $(a_i)\in K_+^{\N}$ such that $|a_i|r^i\to 0\ (i\to+\infty)$. Recall that a non-archimedean analogue of Hadamard formula for the radius of convergence holds, that is, for $(a_i)\in \KX$, we have $(a_i)\in \KXr$ if and only if $1/limsup_{i\in\N_{\ge 1}} |a_i|^{1/i}$, where we set $|0|^{1/i}=0$ for $i\in\N_{\ge 1}$ and $1/+\infty=0$ and $1/0=+\infty$ (see \cite[Chapter 6, \S 1, Proposition 1]{robert}: although $K$ is assumed to be a subfield of $\mathbb{C}_p$ in the reference, the proof is valid for all $K$).

Let $\partial_X$ be the derivation on $\KX$ defined by $\partial_X((a_i))=((i+1)a_{i+1})$. The derivation $\partial_X$ induces derivations on $\KXrz,\KXr,\KXrp$, which we denote by $\partial_X$ for simplicity.

We denote the left $\KXT$-module $L(\KX,\partial_X)$ by $L_0$. Similarly we denote the left $\KXrT$-module $L(\KXr,\partial_X)$ (resp. $\KXrzT$-module $L(\KXrz,\partial_X)$, $\KXrpT$-module $L(\KXrp,\partial_X)$) by $L_r$ (resp. $L_{r,bd},L_{r+}$).


\subsection{Kedlaya-Xiao morphism}

We recall a ring homomorphism defined by Kedlaya-Xiao in \cite{kx} and give its variants.

We define the ring homomorphism $g_0:K\to \KX$ by $g_0(c)=(\partial^i(c)/i!)_i$. 

\begin{lemma}\label{kx}
\begin{enumerate}
\item For $r\in (0,+\infty)$, we have $\KXrz^{\times}=\KXr^{\times}$.
\item For $c\in K$, we have $g_0(c)\in \KXrKz$.
\item For all $r\in (0,\rK]$ and $c\in K$, we have $|c|=|(\partial^i(c)/i!)|_r$.\item For $i\in\N$, we have $|\partial^i/i!|_{op,K}\le 1/r(K,\partial)^i$.
\end{enumerate}
\end{lemma}

\begin{proof}
Part 1 is a basic fact. For parts 2,3,4, see \cite[Lemma 1.2.12, Corollary 1.2.13]{kx}. We give quick proofs.

1. Let $y\in (\KXr)^{\times}$. We have $sup_{s\in (0,r)}|y|_s=sup_{s\in [r/2,r)}|y|_s=sup_{s\in [r/2,r)} |y^{-1}|^{-1}_s=sup_{s\in [r/2,r)} |y^{-1}|_s^{-1}\le |y^{-1}|_{r/2}^{-1}$. Hence $y\in \KXrz$.

2. For $x\in K^{\times}$, we have $g_0(x)\in \KXrK$ by non-archimedean Hadamard formula, which implies $x\in \KXrKz$ by part 1.

3. For $c\in K^{\times}$, $|c^{\pm}|\le |g_{0}(c^{\pm})|_r$. Hence $|c|\ge |g_{0}(c)|_r\ge |c|$. 


4. It follows from part 3 with $r=\rK$.
\end{proof}

The ring homomorphism $g_0$ induces ring hom.'s $K\to \KXrz,K\to \KXr$ for $r\in (0,r(K,\partial)]$ and $K\to \KXrp$ for $r\in (0,r(K,\partial))$, which is denoted by $g_{r,bd},g_r,g_{r+}$ respectively. Since these ring homomorphisms commute with the $\partial$ and $\partial_X$, we define the ring homomorphisms $f_0:\KT\to \KXT,f_{r,bd}:\KT\to \KXrzT,f_r:\KT\to \KXrT,f_{r+}:\KT\to \KXrpT$ for $r\in (0,r(K,\partial)]$ and  for $r\in (0,r(K,\partial))$ as the ones associated to $g_0,g_{r,bd},g_r,g_{r+}$ respectively. Thus we obtain the left $\KT$-modules $f_{0*}L_0,f_{r,bd*}L_{r,bd},f_{r*}L_r$ for $r\in (0,r(K,\partial)]$ and $f_{r+*}L_{r+}$ for $r\in (0,r(K,\partial))$. Note that we have, as subobjects of $f_{r*}L_r=\cap_{t\in (0,r)}f_{t,bd*}L_{t,bd}$ for $r\in (0,r(K,\partial)]$ and $f_{r+*}L_{r+}=\cup_{t\in (r,\rK)}f_{t,bd*}L_{t,bd}=\cup_{t\in (r,\rK)}f_{t*}L_{t}$ for $r\in (0,r(K,\partial))$.

We write $f_{0*}L_0=(K_+^{\N},\rho:\KT\to \E_{\A}(K_+^{\N}))$. Then $\rho(i(c))((a_i)_i)=(\sum_{i=j+k}(\partial^k(c)/k!)a_j)_i$ for $c\in K$ and $\rho(T)((a_i)_i)=((i+1)a_{i+1})_i$. By definition, the left $\KT$-modules $f_{r,bd*}L_{r,bd},f_{r*}L_r$ for $r\in (0,r(K,\partial)]$ and $f_{r+*}L_{r+}$ for $r\in (0,r(K,\partial))$ are submodules of $f_{0*}L_0$, where the underlying abelian groups coincide with those of $\KXrz,\KXr,\KXrp$ respectively.

\noindent Convention. In the rest of the paper except \S \ref{decJ}, when we consider $\KXr,\KXrz$, we tacitly assume that $r\in (0,\rK]$; when we consider $\KXrp$, we tacitly assume that $r\in (0,\rK)$.



\subsection{The representability of the dual functor by $f_{0*}L_0$}

\begin{lemma}\label{dual lemma}
\begin{enumerate}
\item For a morphism $\alpha:M'\to M$ and $x\in\KT$, the diagram in $\A$ below commutes.
$$
\xymatrix{
\H(i_*M,R)\ar[rr]^{\H(i_*\alpha,R)}\ar[d]^{\rho_{D_0M}(x)}&&\H(i_*M',R)\ar[d]^{\rho_{D_0M}(x)}\\
\H(i_*M,R)\ar[rr]^{\H(i_*\alpha,R)}&&\H(i_*M',R).
}
$$
\item For $M\in\C$ and $x\in \KT$, the diagram in $\C$ below commutes.
$$
\xymatrix{
M\ar[r]^{c_{0,M}}\ar[d]^{\rho_M(x)}&D_0D_0M\ar[d]^{\rho_{D_0D_0M}(x)}\\
M\ar[r]^{c_{0,M}}&D_0D_0M.
}
$$
\item For a morphism $\alpha:M'\to M$ and $x\in\KT$, the diagram in $\C$ below commutes.
$$
\xymatrix{
M'\ar[r]^{c_{0,M'}}\ar[d]^{\alpha}&D_0D_0M'\ar[d]^{D_0D_0\alpha}\\
M\ar[r]^{c_{0,M}}&D_0D_0M.
}
$$
\item Let $M\in\C,\chi\in D_{0}M$. The morphism $\Psi_M(\chi):M_+\to K_+^{\N};x\mapsto (\chi(\frac{1}{i!}T^i\cdot x))_i$ in $\A$ defines a morphism $M\to f_{0*}L_0$ in $\C$.
\end{enumerate}
\end{lemma}
\begin{proof}
To prove parts 1,2, we have only to prove the assertion for $x\in i(K)\cup\{T\}$, in which case the assertion follows from a straightforward calculation. Part 3 is obvious. 

We prove part 4. We have only to prove that for $x\in M$ and $c\in i(K)\cup\{T\}$, we have $\Psi_M(\chi)(c\cdot x)=c\cdot (\Psi_M(\chi)(x))$; if this is the case then $\Psi_M(\chi)(c\cdot x)=c\cdot (\Psi_M(\chi)(x))$ for $x\in M$ and $c\in \KT$. Let $c\in K$. We have $\Psi_M(\chi)(i(c)\cdot x)=(\chi(\frac{1}{i!}T^i\cdot (i(c)\cdot x)))$. For $i\in\N$, we have $\chi(\frac{1}{i!}T^i\cdot (i(c)\cdot x))=\chi((\frac{1}{i!}T^i\cdot i(c))\cdot x)=\chi(\sum_{i=j+k}(\frac{1}{i!}\binom{i}{j}i(\partial^k(c))T^j)\cdot x)=\chi(\sum_{i=j+k}(i(\frac{1}{k!}\partial^k(c))\cdot (\frac{1}{j!}T^j)\cdot x))=\sum_{i=j+k}\frac{1}{k!}\partial^k(c)\cdot \chi((\frac{1}{j!}T^j)\cdot x)$. Hence $\Psi_M(\chi)(i(c)\cdot x)=i(c)\cdot \Psi_M(\chi)(x)$. We have $\Psi_M(\chi)(T\cdot x)=(\chi(\frac{1}{i!}T^i\cdot (T\cdot x)))=(\chi(\frac{i+1}{(i+1)!}T^{i+1}\cdot x))=((i+1)\chi(\frac{1}{(i+1)!}T^{i+1}\cdot x))=T\cdot (\Psi_M(\chi)(x))$.
\end{proof}

\begin{definition}
The function $D_0$ defines a contravariant endofunctor on $\C$ by Lemma \ref{dual lemma}. Note that the composition $D_0D_0$ of $D_0$ followed by $D_0$ defines the covariant endofunctor on $\C$. By Lemma \ref{dual lemma} again, the function $c_0$ defines a natural transformation $id_{\C}\to D_0D_0$.
\end{definition}

\begin{definition}
\begin{enumerate}
\item  We define the functor $D_+:\C\to\A$ by $D_+=\H(-,f_{0*}L_0)$. We define the morphism $\theta:i_{0*}f_{0*}L_0\to K$ in $Mod(K)$ by the morphism $K_+^{\N}=(i_{0*}f_{0*}L_0)_+\to K_+;(a_i)\mapsto a_0$ in $\A$. 

\item We define the natural transformation $\Phi:D_+\to (D_0)_+$ as the following vertical composition $\Phi=\H(i_*-,\theta)\cdot i_*:\H(-,f_{0*}L_0)\to \H(i_*-,i_*f_{0*}L_0)\to \H(i_*-,K)$. 

\item We define the natural transformation $\Psi:(D_0)_+\to D_+$ by the function as in Part 4 of Lemma \ref{dual lemma}, which forms a natural transformation as we can see easily.
\end{enumerate}
\end{definition}

\begin{lemma}\label{abelian natural isom}
The natural transformation $\Phi,\Psi$ are natural isomorphisms with $\Phi^{-1}=\Psi$.
\end{lemma}

\begin{proof}
Let $M\in \C$. Let $\chi\in (D_0)_+M$. Let $x\in M$. Then $(\Phi\cdot\Psi)_M(\chi)(x)=(\Phi_M\circ \Psi_M)(\chi)(x)=\Phi_M (\Psi_M(\chi))(x)=\theta(\Psi_M(\chi)(x))=\chi(\frac{1}{0!}T^0\cdot x)=\chi(x)$. Hence $\Phi_M\circ \Psi_M=id_{D_0M}$.

Let $s\in D_+M$. Let $x\in M$. Then $(\Psi\cdot\Phi)_M(s)(x)=(\Psi_M\circ \Phi_M)(s)(x)=\Psi_M (\Phi_M(s))(x)=(\Phi_M(s)(\frac{1}{i!}T^i\cdot x))=(\theta(s(\frac{1}{i!}T^i\cdot x)))=(\theta(\frac{1}{i!}T^i\cdot (s(x))))=s(x)$. Hence $\Psi_M\circ \Phi_M=id_{DM}$. 
\end{proof}

\begin{corollary}\label{injective cogenerator}
The functor $(D_{0})_+$ is represented by $f_{0*}L_0$. In particular, the left $\KT$-module $f_{0*}L_0$ is an injective cogenerator in $\C$.

\end{corollary}

\begin{definition}
Let $\E((D_0)_+),\E(D_+)$ denote the classes of natural transformations of $(D_0)_+,D_+$ respectively. By Lemma \ref{abelian natural isom} and Yoneda's lemma, $\E((D_0)_+),\E(D_+)$ form rings, where the multiplications are given by vertical composition. By Lemma \ref{dual lemma}, the map $\rhozt:\KT\to\E((D_0)_+);x\mapsto (M\mapsto (\rho_{D_0M}(x):(D_0)_+M\to (D_0)_+M))$ is a ring homomorphism. We define the ring homomorphism $\rhot:\KT\to\E(f_{0*}L_0)$ by $\rhot(c)=\Psi\cdot\rhozt(c)\cdot\Phi$ for $c\in\KT$. We define the ring homomorphism $\rhos:\KT\to\E(f_{0*}L_0)$ by $\rhos(c)=\rhot(c)_{f_{0*}L_0}(id_{f_{0*}L_0})$ for $c\in\KT$. Then $(f_{0*}L_0,\rhos)$ is an object of $\CC$, which is denoted by $f_{0*}L_0$ for simplicity. We define the contravariant endofunctor $D$ on $\C$ by $D(-)=\H(-,f_{0*}L_0)$.
\end{definition}

\begin{lemma}
For $(a_i)_i\in K_+^{\N}$, we have $\rhos(i(c))((a_i))=(ca_i)$ for $c\in K$ and $\rhos(T)((a_i))=(\partial(a_i)-(i+1)a_{i+1})$.
\end{lemma}
\begin{proof}
For $x\in\KT$, we have $\rhos(x)=\Psi_{f_{0*}L_0}\circ\rhozt(x)_{f_{0*}L_0}\circ\Phi_{f_{0*}L_0}(id_{f_{0*}L_0})$. Let $c\in K$. We have $\rhos(i(c))((a_i))=\Psi_{f_{0*}L_0}\circ\rhozt(i(c))_{f_{0*}L_0}\circ\Phi_{f_{0*}L_0}(id_{f_{0*}L_0})((a_i))=(\rhozt(i(c))_{f_{0*}L_0}(\Phi_{f_{0*}L_0}(id_{f_{0*}L_0}))(\frac{1}{i!}T^i\cdot (a_i)))=(\rhozt(i(c))_{f_{0*}L_0}(\theta)(\frac{1}{i!}T^i\cdot (a_i)))=(c\cdot (\theta(\frac{1}{i!}T^i\cdot (a_i))))=(c\cdot a_i)$. We have $\rhos(T)((a_i))=\Psi_{f_{0*}L_0}\circ\rhozt(T)_{f_{0*}L_0}\circ\Phi_{f_{0*}L_0}(id_{f_{0*}L_0})((a_i))=(\rhozt(T)_{f_{0*}L_0}(\Phi_{f_{0*}L_0}(id_{f_{0*}L_0}))(\frac{1}{i!}T^i\cdot (a_i)))=(\rhozt(T)_{f_{0*}L_0}(\theta)(\frac{1}{i!}T^i\cdot (a_i)))=(\partial(\theta(\frac{1}{i!}T^i\cdot (a_i))-\theta(T\cdot\frac{1}{i!}T^i\cdot (a_i)))=(\partial(\theta(\frac{1}{i!}T^i\cdot (a_i))-\theta(\frac{i+1}{(i+1)!}T^{i+1}\cdot (a_i)))=(\partial(a_i)-(i+1)a_{i+1})$.
\end{proof}

\begin{lemma}\label{natural isom}
For $M\in \C$, the isomorphisms $\Phi_M:D_+M\to (D_0)_+M,\Psi_M:(D_0)_+M\to D_+M$ define morphisms  $DM\to D_0M,D_0M\to DM$ respectively, which are denoted by $\Phi_M$ and $\Psi_M$ respectively. Furthermore $\Phi_M:DM\to D_0M$ and $\Psi_M:D_0M\to DM$ are isomorphisms and inverse to each other. The functions $\Phi$ and $\Psi$ defines natural isomorphisms $\Phi:D\to D_0$ and $\Psi:D_0\to D$ respectively.
\end{lemma}

\begin{proof}
It follows from the construction.
\end{proof}



\subsection{The natural transformation $c:id_{\C} \to DD$}

\begin{definition}
Let $M\in\C$. We define the morphisms $(\Phi\circ\Psi)_M:DDM\to D_0D_0M,(\Psi\circ\Phi)_M:D_0D_0M\to DDM$ in $\C$ by $(\Phi\circ\Psi)_M=\Phi_{D_0M}\circ D\Psi_M=D_0\Psi_M\circ\Phi_{DM},(\Psi\circ\Phi)_M=\Psi_{DM}\circ D_0\Phi_M=D\Phi_M\circ\Psi_{D_0M}$ in $\C$. Then the above functions define the natural transformations $\Phi\circ\Psi:DD\to D_0D_0,\Psi\circ\Phi:D_0D_0\to DD$, which are inverse each other.
\end{definition}

\noindent Notation. For a left $\KT$-module $M$ and $s\in DM,x\in M$, we denote $s(x)\in M'$ by $\langle s,x\rangle=(\langle s,x\rangle_i)_i$.


\begin{lemma}\label{explicit}
\begin{enumerate}
\item Let $M\in\C,\chi\in D_0M,x\in M,i\in\N$. Then we have $(\frac{1}{i!}T^i\cdot\chi)(x)=\sum_{i=j+k}\frac{(-1)^j}{k!}\partial^k(\chi(\frac{1}{j!}T^j\cdot x))$.

\item Let $M\in \C$. Let $x\in M,s\in DM$. We have $\langle c_M(x),s\rangle=(\sum_{i=j+k}\frac{(-1)^j}{k!}\partial^k(\langle s,x\rangle_j))_i$.
\end{enumerate}
\end{lemma}

\begin{proof}
1. We prove by induction on $i$. In the base case $i=0$, we have nothing to prove.  In the case $i=1$, we have $(\frac{1}{1!}T\cdot\chi)(x)=\partial(\chi(x))-\chi(T\cdot x)$ by the definition of $D_0$. In the induction step, 

$(\frac{1}{(i+1)!}T^{i+1}\cdot\chi)(x)=\frac{1}{i+1}(\frac{1}{i!}T^i\cdot(T\cdot\chi))(x)=\frac{1}{i+1}\sum_{i=j+k}\frac{(-1)^j}{k!}\partial^k((T\cdot\chi)(\frac{1}{j!}T^j\cdot x))=\frac{1}{i+1}\sum_{i=j+k}\frac{(-1)^j}{k!}\partial^k(\{\partial(\chi(\frac{1}{j!}T^j\cdot x))-\chi(T\cdot \frac{1}{j!}T^j\cdot x)\})=\frac{1}{i+1}\sum_{i=j+k}\frac{(-1)^j(k+1)}{(k+1)!}\partial^{k+1}(\chi(\frac{1}{j!}T^j\cdot x)-\frac{1}{i+1}\sum_{i=j+k}\frac{(-1)^j}{k!}\partial^{k}(\chi(\frac{j+1}{(j+1)!}T^{j+1}\cdot x))=\frac{1}{i+1}\sum_{i+1=j+k}\frac{(-1)^jk}{k!}\partial^{k}(\chi(\frac{1}{j!}T^j\cdot x))-\frac{1}{i+1}\sum_{i+1=j+k}\frac{(-1)^{j-1}j}{k!}\partial^{k}(\chi(\frac{1}{j!}T^{j}\cdot x))=\sum_{i+1=j+k}\frac{(-1)^j}{k!}\partial^{k}(\chi(\frac{1}{j!}T^j\cdot x))$.

2. We have $\langle c_{M}(x),s\rangle=\langle (\Psi\circ\Phi)_M\circ (c_{0,M}(x)),s\rangle=\langle (D\Phi_M\circ\Psi_{D_0M})(c_{0,M}(x)),s\rangle=\langle D\Phi_M(\Psi_{D_0M}(c_{0,M}(x))),s\rangle=\langle \Psi_{D_0M}(c_{0,M}(x)),\Phi_{M}(s)\rangle=(c_{0,M}(x)(\frac{1}{i!}T^i\cdot (\Phi_M(s))))_i$.

Let $i\in\N$. We have $c_{0,M}(x)(\frac{1}{i!}T^i\cdot (\Phi_M(s)))=(\frac{1}{i!}T^i\cdot (\Phi_M(s)))(x)=\sum_{i=j+k}\frac{(-1)^j}{k!}\partial^k(\Phi_M(s)(\frac{1}{j!}T^j\cdot x))=\sum_{i=j+k}\frac{(-1)^j}{k!}\partial^k(\theta(s(\frac{1}{j!}T^j\cdot x))=\sum_{i=j+k}\frac{(-1)^j}{k!}\partial^k(\theta(\frac{1}{j!}T^j\cdot (s(x))))=\sum_{i=j+k}\frac{(-1)^j}{k!}\partial^k(\langle s,x\rangle_j)$.
\end{proof}

\begin{lemma}\label{switch}
\begin{enumerate}
\item For $M\in \C,x\in M,s\in DM$, we have $\langle s,x\rangle=\langle c_{DM}(s),c_M(x)\rangle$.

\item For $M\in \C,x\in M,s\in DM$, \TFAE

\begin{enumerate}
\item We have $\langle s,x\rangle=0$.

\item We have $\langle c_M(x),s\rangle=0$.
\end{enumerate}

\end{enumerate}
\end{lemma}

\begin{proof}
1. We have $\langle c_{DM}(s),c_M(x)\rangle=(\sum_{i=j+k}\frac{(-1)^j}{k!}\partial^k(\langle c_M(x),s\rangle_j))$.

Let $i\in\N$. We have $\sum_{i=j+k}\frac{(-1)^j}{k!}\partial^k(\langle c_M(x),s\rangle_j)=\sum_{i=j+k}\frac{(-1)^j}{k!}\partial^k(\sum_{j=g+h}\frac{(-1)^g}{h!}\partial^h(\langle s,x\rangle_g))=\sum_{i=j+k}\sum_{j=g+h}\frac{(-1)^{j+g}}{k!h!}\partial^{k+h}(\langle s,x\rangle_g)=\sum_{i=j+k}\sum_{j=g+h}(-1)^{h}\binom{k+h}{h}\frac{1}{(k+h)!}\partial^{k+h}(\langle s,x\rangle_g)=\sum_{i=f+g}\sum_{f=h+k}(-1)^{h}\binom{f}{h}\frac{1}{f!}\partial^f(\langle s,x\rangle_g)=\sum_{i=f+g}0^f\frac{1}{f!}\partial^f(\langle s,x\rangle_g)=\langle s,x\rangle_i$.

Hence $\langle c_{DM}(s),c_M(x)\rangle=\langle s,x\rangle$.

2. For any $M$, (a)$\Rightarrow$(b) holds by Lemma \ref{explicit}.

$(b)(a)$ Assume $\langle c_M(x),s\rangle=0$. By applying (a)$\Rightarrow$(b) to $DM$, we have $\langle c_{DM}(s),c_M(x)\rangle=0$. Hence $\langle s,x\rangle=\langle c_{DM}(s),c_M(x)\rangle=0$ by part 1.

\end{proof}


\begin{lemma}\label{duality}
\begin{enumerate}
\item The natural transformations $c_0,c$ are pointwise monomorphisms.
\item The functors $D_0,D$ are faithful and exact. Moreover $D_0,D$ preserve dimension.
\item Let $D^f_0,D^f$ denote the endfunctors on $\Cf$ induced by $D_0,D$ respectively. Let $c^f_0:id_{\Cf}\to D^f_0D^f_0,c^f:id_{\Cf}\to D^fD^f$ denote the natural transformations induced by $c_0,c$ respectively. Then $c^f_0,c^f$ are natural isomorphisms. 
\item The functors $D^f,D^f_0$ preserve length. In particular, for $M\in\Cf$, $M$ is irreducible if and only if so is $D^f_0M$ or $D^fM$.
\end{enumerate}
\end{lemma}

\begin{proof}
Since $D_0,D$ are naturally isomorphic and we have $c=(\Psi\circ\Phi)\cdot c_0$ and $\Psi\circ\Phi$ is a natural isomorphism, we have only to prove the assertion for $D_0,c_0$. Then one can prove easily by using the forgetful functor $\C\to Mod(K)$ appropriately.
\end{proof}

\begin{lemma}\label{functoriality lemma}
Let $\alpha:M'\to M$ be a morphism in $\C$, $x'\in M',s\in DM$. Then we have $\langle c_M(\alpha(x')),s\rangle=\langle c_{M'}(x'),s\circ\alpha\rangle$.
\end{lemma}
\begin{proof}
Since $c$ is a natural transformation, $DD\alpha\circ c_{M'}=c_M\circ\alpha$. Hence the assertion follows from $\langle DD\alpha\circ c_{M'}(x'),s\rangle=\langle c_{M'}(x'),s\circ\alpha\rangle,\langle c_M\circ\alpha (x'),s\rangle=\langle c_{M}(\alpha(x')),s\rangle$.
\end{proof}

\section{The functors $F_{\pm}$}

We construct two endofunctor $F_{\pm}$ on $\C$ associated to a subobject of $f_{0*}L_0$. We give examples of such subobjects. First examples are certain $p$-adic Banach spaces, whose construction is based on an idea of Robba (\cite{robba}). Second examples are given by push-outs of differential rings, which is related to Dwork-Robba's work (\cite{dr}) Then we calculate the associated functors $\Fp,\Fm$ at some level.



\subsection{Subobjects of $f_{0*}L_0$}

In this section, we consider the object $M_{\lambda}$ satisfying condition (ST).

(ST) $M_{\lambda}$ is a submodule of the left $\KT$-module $f_{0*}L_0$ which is stable under the action of $\rho(x)$ for $x\in \KT$.

Note that we may replace the condition $x\in \KT$ by $x\in i(K)\cup\{T\}$ since $\KT$ is generated by the subset $i(K)\cup\{T\}$ as a ring. Assume $M_{\lambda}$ satisfies (ST). Then $M_{\lambda}$ is regarded as an object of $\C$ in an obvious way. Moreover $M_{\lambda}$ is regarded as an object of $\CC$ with respect to the ring homomorphism $\KT\to\E(M_{\lambda})$ induced by $\rho$. Furthermore the canonical injection $M_{\lambda}\to f_{0*}L_0$ in $\C$ is a morphism in $\CC$, which is again a monomorphism in $\CC$. Thus we obtain a monomorphism $M_{\lambda}\to f_{0*}L_0$ in $\CC$.



\subsection{Definition of $F_{\pm}$}\label{fil}

\begin{definition}
\begin{enumerate}
\item For $M\in\C$, we define the object $\Fm M$ in $\C$ by $\Fm M=\{x\in M;\forall s\in \H(M,M_{\lambda}),c_M(x)(s)=0\}$. For a morphism $\alpha:M'\to M$ in $\C$, by using Lemma \ref{functoriality lemma}, $\alpha$ induces a unique morphism $\Fm M'\to \Fm M$ of subobjects of $M$, which is denoted by $\Fm \alpha$. The function $\Fm$ forms an endofunctor on $\C$. We have an obvious natural transformation $I_{-}:\Fm\to id_{\C}$ by definition, which is a pointwise monomorphism.
\item For $M\in\C$, we define the object $\Fp M$ in $\C$ by $\Fp M=\{x\in M;\forall s\in \H(M,f_{0*}L_0),c_M(x)(s)\in M_{\lambda}\}$. For a morphism $\alpha:M'\to M$ in $\C$, by using Lemma \ref{functoriality lemma}, $\alpha$ induces a unique morphism $\Fp M'\to \Fp M$ of subobjects of $M$, which is denoted by $\Fp \alpha$. The function $\Fp$ forms an endofunctor on $\C$. We have an obvious natural transformation $I_{+}:\Fp\to id_{\C}$ by definition, which is a pointwise monomorphism.
\end{enumerate}
\end{definition}

Note that for $M\in\C$, we have the exact sequences in $\C$
$$
\xymatrix{
0\ar[r]& \H(M,M_{\lambda})\ar[r]^(.7){\alpha_1}& DM\ar[r]^(.3){\alpha_2}& DM/\H(M,M_{\lambda})\ar[r]&0,
}
$$
$$
\xymatrix{
0\ar[r]& \H(DM,M_{\lambda})\ar[r]^{\beta_1}& DDM\ar[r]^(.3){\beta_2}& \H(DM,f_{0*}L_0/M_{\lambda}),
}
$$
where $\alpha_1,\beta_1$ are induced by the morphism $M_{\lambda}\to f_{0*}L_0$ in $\CC$, $\alpha_2$ is the canonical surjection, $\beta_2$ is induced by the obvious morphism $f_{0*}L_0\to f_{0*}L_0/M_{\lambda}$ in $\CC$. Moreover we have the exact sequence
$$
\xymatrix{
0\ar[r]& D(DM/\H(M,M_{\lambda}))\ar[r]^(.7){D\alpha_2}&DDM\ar[r]^{D\alpha_1}& D\H(M,M_{\lambda}).
}$$


\begin{proposition}\label{dimension formula}
\begin{enumerate}
\item For $M\in\C$, we have $\Fm M=\cap_{s\in \H(M,M_{\lambda})}\ker{s}$.

\item For $M\in\C$, there exists a unique morphism $c_{-,M}:\Fm M\to D(DM/\H(M,M_{\lambda}))$ making the square commutative in the following diagram.

$$
\xymatrix{
&F_-M\ar[d]^{c_{+,M}}\ar[r]^{I_{-,M}}&M\ar[d]^{c_M}&&\\
0\ar[r]&D(DM/\H(M,M_{\lambda}))\ar[r]^(.7){D\alpha_1}& DDM\ar[r]^(.3){D\alpha_2}&D\H(M,M_{\lambda})\ar[r]&0
}
$$

Moreover the square is a pull-back.

\item For $M\in\C$, there exists a unique morphism $c_{+,M}:\Fp M\to \H(DM,M_{\lambda})$ such that the following diagram commutative.

$$
\xymatrix{
&F_+M\ar[d]^{c_{+,M}}\ar[r]^{I_{+,M}}&M\ar[d]^{c_M}&\\
0\ar[r]&\H(DM,M_{\lambda})\ar[r]^{\beta_1}& DDM\ar[r]^(.3){\beta_2}&\H(DM,f_{0*}L_0/M_{\lambda}).
}
$$

Moreover the square is a pull-back.

\item For $M\in\Cf$, the morphisms $c_{+,M},c_{-,M}$ are isomorphisms.

\item For $M\in\Cf$, we have $dim \Fm M+dim \H(M,M_{\lambda})=dim M$ and $dim \Fp M=dim \H(DM,M_{\lambda})$.

\item Let $M\in\Cf$. Let $J_{-,M}:M\to M/\Fm M$ be the canonical surjection. Then $\H(J_{-,M},M_{\lambda}):\H(M/\Fm M,M_{\lambda})\to \H(M,M_{\lambda})$ is an isomorphism and the morphism $\H(I_{-,M},M_{\lambda}):\H(M,M_{\lambda})\to \H(\Fm M,M_{\lambda})$ is equal to $0$.

\item We have $\Fm (M/\Fm M)=0$ and $dim \H(M/\Fm M,M_{\lambda})=dim M/\Fm M$.
\end{enumerate}
\end{proposition}

\begin{proof}
1. By Lemma \ref{switch}, $\Fm M=\{x\in M;\forall s\in \H(M,M_{\lambda}),<c_M(x),s>=0\}=\{x\in M;\forall s\in \H(M,M_{\lambda}),<s,x>=0\}=\cap_{s\in \H(M,M_{\lambda})}\ker{s}$.

2. The uniqueness is obvious. We can see that the underlying sets $\Fm M$ coincides with the kernel of $D\alpha_2\circ c_M$, which implies the rest of assertion.

3. The uniqueness is obvious. We can see that the underlying sets $\Fp M$ coincides with the kernel of $\beta_2\circ c_M$, which implies the rest of assertion.

4. This is a consequence of parts 2 and 3. 

5. It follows from the left exactness of the functor $\H(-,M_{\lambda})$ and part 1.

6. By applying part 3 to $M/\Fm M$ and $M$ and using part 4, we have $dim \Fm (M/\Fm M)=dim M/\Fm M-dim \H(M/\Fm M,M_{\lambda})=dim M/\Fm M-dim \H(M,M_{\lambda})=0$.

7. It follows from parts 5 and 6.
\end{proof}



\subsection{Example 1: Robba's construction}\label{ex1}

Let $\pi=(\pi_i)$ be a sequence of $\R_{>0}$ satisfying condition (C).

(C) The sequence $(\pi_{i+1}/\pi_i)$ is non-decreasing sequence and $sup \pi_{i+1}/\pi_i\le r(K,\partial)$.

For example, $\pi(t)=(t^i)$ for $t\in (0,r(K,\partial)]$ satisfies the condition(see \cite{robba}).

\begin{lemma}\label{inequality}
We have $\pi_i\pi_j/(r(K,\partial)^h\pi_0\pi_k)\le 1$ for $h,i,j,k\in\N$ such that $k\ge j$ and $i=h+k-j$.
\end{lemma}

\begin{proof}
We have $\pi_i=(\pi_i/\pi_{i-1})\dots (\pi_{i-h+1}/\pi_{i-h})\pi_{i-h}\le r(K,\partial)^h\pi_{i-h}=r(K,\partial)^h\pi_{k-j}$. If $j>0$ then we have $\pi_j/\pi_k\le \pi_{j-1}/\pi_{k-1}$ by $\pi_j/\pi_k=(\pi_{k-1}/\pi_k) (\pi_j/\pi_{k-1})\le (\pi_{k-2}/\pi_{k-1}) (\pi_{j-1}/\pi_{k-2})=\pi_{j-1}/\pi_{k-1}$. Repeating a similar argument, $\pi_j/\pi_k\le\pi_0/\pi_{k-j}$. Hence $\pi_i \pi_j/\pi_k\le r(K,\partial)^h\pi_0$. 
\end{proof}

Let $t_{\rK,bd}:\KXrKzT\to \KXT$ be the ring homomorphism associated to the obvious ring homomorphism $\KXrKz\to \KX$. Let $i_{\rK,bd}:\KXrKz\to \KXrKzT;x\mapsto xT^0$ be the ring homomorphism as before. Recall that the underlying abelian group of the left $\KT$-module $f_*L_0$ is given by  the direct product $K_+^{\N}$.

\begin{lemma}\label{stable lemma}
Let $B_{\pi}$ be the subgroup of $K_+^{\N}$ defined by $\{(a_i)\in K_+^{\N};sup |a_i|\pi_i<+\infty\}$. Then $B_{\pi}$ is a submodule of $t_{\rK,bd*}L_0$.
\end{lemma}

\begin{proof}
Let $(a_i)\in B_{\pi}$. For $(c_i)\in \KXrKz$, $|\sum_{i=j+k}c_ka_j|\pi_i\le max_{i=j+k}|c_k||a_j|\pi_i\le max_{i=j+k}|c_k|\rK^k|a_j|\pi_j\le sup_{k\in \N}|c_k|\rK^k sup_{j\in\N}|a_j|\pi_j$. Hence $i_{\rK,bd}((c_i))\cdot (a_i)\in B_{\pi}$. We have $|(i+1)a_{i+1}|\pi_i\le |a_{i+1}|\pi_{i+1}(\pi_i/\pi_{i+1})\le |a_{i+1}|\pi_{i+1}(\pi_0/\pi_1)$. Hence $T\cdot (a_i)\in B_{\pi}$. Since $\KXrKzT$ is generated by the subset $i_{\rK,bd}(\KXrKz)\cup\{T\}$, $x\cdot (a_i)\in B_{\pi}$ for $x\in \KXrKzT$.
\end{proof}

\begin{definition}
We define $N_{\pi}\in \CKXrKzT$ as the submodule of $L_0$ given by $B_{\pi}$. The left $\KT$-module $f_{\rK,bd*}N_{\pi}$ is a submodule of $f_{0*}L_0$ since $f_{\rK,bd*}N_{\pi}$ is regarded as a submodule of $f_{\rK,bd*}t_{\rK,bd*}L_0=f_{0*}L_0$.
\end{definition}

\begin{lemma}\label{ST1}
Let $(a_i)\in f_*N_{\pi}$. We have $\rhos(x)((a_i))\in f_*N_{\pi}$ for $x\in \KT$.
\end{lemma}

\begin{proof}
By definition, we have $\rhos(i(c))((a_i))=(ca_i)\in f_*N_{\pi}$ for $c\in K$ and $\rhos(T)((a_i))=(\partial(a_i)-(i+1)a_{i+1})$. For $i\in\N$, $|\partial(a_i)-(i+1)a_{i+1}|\pi_i\le max(|\partial(a_i)|,|(i+1)a_{i+1}|)\pi_i\le max(|\partial|_{op,K}|a_i|\pi_i,|a_{i+1}|\pi_i)\le max(|\partial|_{op,K}|a_i|\pi_i,|a_{i+1}|\pi_{i+1}(\pi_0/\pi_1))$. Hence we have $sup_i |\partial(a_i)-(i+1)a_{i+1}|\pi_i\le max(|\partial|_{op,K},\pi_0/\pi_1) sup_i |a_i|\pi_i<+\infty$. Therefore $\rhos(T)((a_i))\in f_*N_{\pi}$. Since $\rhos$ is a ring homomorphism and $\KT$ is generated by the subset $i(K)\cup\{T\}$, we have $\rhos(x)((a_i))\in f_*N_{\pi}$ for $x\in \KT$.
\end{proof}


\begin{definition}
Let $\pi$ be as above. We define the endofunctors $\Fpim,\Fpip$ on $\C$ as $\Fm$ and $\Fp$ obtained by applying the result of 3.1 to $f_{\rK,bd*}N_{\pi}$.
\end{definition}


\begin{lemma}\label{switch pi}
Let $\pi$ be as above. For $M\in \C,x\in M$ and $s\in DM$, \TFAE.

(a) We have $\langle s,x\rangle\in f_*N_{\pi}$.

(b) We have $\langle c_M(x),s\rangle\in f_*N_{\pi}$.
\end{lemma}

\begin{proof}
For any $M$, we prove (a)$\Rightarrow$(b). Assume $\langle s,x\rangle\in f_*N_{\pi}$. We set $C:=sup_i |\langle s,x\rangle_i|\pi_i$. Let $i\in \N$. By Lemma \ref{kx} and \ref{explicit} and \ref{inequality}, we have $|\langle c_M(x),s\rangle_i|\pi_i\le |\sum_{i=j+k}\frac{(-1)^j}{k!}\partial^k(\langle s,x\rangle_j)|\pi_i\le max_{i=j+k}|\frac{1}{k!}\partial^k(\langle s,x\rangle_j)|\pi_i\le max_{i=j+k}\frac{1}{r(K,\partial)^k}|\langle s,x\rangle_j|\pi_i\le max_{i=j+k}|\langle s,x\rangle_j|\pi_j\le C$. Hence $sup |\langle c_M(x),s\rangle_i|\pi_i\le C$. Therefore $\langle c_M(x),s\rangle\in f_*N_{\pi}$.

(b)$\Rightarrow$(a): Assume (b) holds. By applying (a)$\Rightarrow$(b) to $DM$, $\langle c_{DM}(s),c_M(x)\rangle\in f_*N_{\pi}$. By Lemma \ref{switch}, $\langle s,x\rangle=\langle c_{DM}(s),c_M(x)\rangle\in f_*N_{\pi}$.

\end{proof}

\begin{corollary}\label{description pi}
Let $\pi$ be as above. For $M\in C$, $\Fpip M=\{x\in M;\forall s\in DM,s(x)\in f_*N_{\pi}\}$.
\end{corollary}



\subsection{Example 2: Subrings of the ring of formal power series}

\begin{lemma}\label{ST2}
1. For $r\in (0,r(K,\partial)]$, we have $f_{\rK,bd*}N_{\pi(r)}=f_{r,bd*}L_{r,bd}$ in $\C$.

2. Let $r\in (0,r(K,\partial)]$ and $(a_i)\in f_{r,bd*}L_{r,bd}$ (resp. $f_{r*}L_{r},f_{r+*}L_{r+}$). We have $\rhos(x)((a_i))\in f_{r,bd*}L_{r,bd}$ (resp. $f_{r*}L_{r},f_{r+*}L_{r+}$) for $x\in \KT$.
\end{lemma}

\begin{proof}
1. The both sides are submodules of $f_{0*}L_0\in \C$, whose underlying sets are given by $B_{\pi(r)}$.

2. In the case of $f_{r,bd*}L_{r,bd}$, it follows from part 1 and Lemma \ref{ST1}. This case implies the assertion in the rest of the case.
\end{proof}


\begin{definition}
1. Let $r\in (0,r(K,\partial)]$. We define the endofunctors $\Fzrh,\Fr$ on $\C$ as $\Fm$ and $\Fp$ for $M_{\lambda}=f_{r*}L_r$.

2. Let $r\in (0,r(K,\partial))$. We define the endofunctor $\Fzr$ on $\C$ as $F_-$ for $M_{\lambda}=f_{r+*}L_{r+}$.

3. Let $r\in (0,r(K,\partial))$. We define the endofunctor $\Frr$ on $\C$ as $\Fzr\cap\Fr$.
\end{definition}

\begin{lemma}\label{limit lemma}
Let $M\in \Cf$. 

\begin{enumerate}
\item For $r\in (0,r(K,\partial)]$, there exists $r_0\in (0,r)$ such that for $t\in [r_0,r)$, the obvious morphism $\H(M,f_{r*}L_r)\to \H(M,f_{t,bd*}L_{t,bd})$ is an isomorphism in $\C$.

\item For $r\in (0,r(K,\partial))$, there exists $r_0\in (r,r(K,\partial))$ such that for $t\in (r,r_0]$, the obvious morphism $\H(M,f_{t*}L_t)\to \H(M,f_{r+*}L_{r+})$ is an isomorphism in $\C$.
\end{enumerate}
\end{lemma}

\begin{proof}
The family $\{\H(M,f_{u,bd*}L_{u,bd})\}_{u\in (0,r)}$ (resp. $\{\H(M,f_{u*}L_u)\}_{u\in [r,r(K,\partial)]}$) of objects in $\C$ forms a projective system (resp. injective system) with respect to the obvious transition morphisms. The obvious morphism $\H(M,f_{r*}L_r)\to lim \{\H(M,f_{u,bd*}L_{u,bd})\}_{u\in (0,r)}$ (resp. $colim \{\H(M,f_{u*}L_u)\}_{u\in [r,r(K,\partial)]}\to \H(M,f_{r+*}L_{r+})$) is an isomorphism since $\H$ commutes with colomit (resp. since $M$ is finite).

Since the transition morphisms of the projective system (resp. injective system) are injective, there exists $r_0\in (0,r)$ (resp. $(r,r(K,\partial)]$) such that the transition morphism of the projective system $\{\H(M,f_{u,bd*}L_{u,bd})\}_{u\in [r_0,r)}$ (resp. the injective system $\{\H(M,f_{u*}L_u)\}_{u\in [r_0,r(K,\partial)]}$) are isomorphisms by $\dim \H(M,f_{u,bd*}L_{u,bd})\le \dim \H(M,f_{0*}L_0)=\dim M$ (resp. $\dim \H(M,f_{u*}L_u)\le \dim \H(M,f_{0*}L_0)=\dim M$). For any $r_0$ with this property, the obvious morphism $\H(M,f_{r*}L_r)\to \H(M,f_{t,bd*}L_{t,bd})$ for $t\in [r_0,r)$ (resp. $\H(M,f_{t*}L_t)\to \H(M,f_{r+*}L_{r+})$ for $t\in [r_0,\rK]$) coincides with the composition $\H(M,f_{r*}L_r)\to lim \{\H(M,f_{u,bd*}L_{u,bd})\}_{u\in (0,r)}\to \H(M,f_{t,bd*}L_{t,bd})$ (resp. $\H(M,f_{t*}L_t)\to colim \{\H(M,f_{u*}L_u)\}_{u\in (r,r(K,\partial)]}\to \H(M,f_{r+*}L_{r+})$, which is an isomorphism.
\end{proof}

\begin{lemma}\label{switch r}

\begin{enumerate}
\item Let $r\in (0,r(K,\partial)]$. For $M\in\C,x\in M$ and $s\in DM$, \TFAE.
\begin{enumerate}
\item We have $\langle s,x\rangle \in f_{r,bd*}L_{r,bd}$.
\item We have $\langle c_M(x),s\rangle \in f_{r,bd*}L_{r,bd}$.
\end{enumerate}
\item Let $r\in (0,r(K,\partial)]$. For $M\in\C,x\in M$ and $s\in DM$, \TFAE.
\begin{enumerate}
\item We have $\langle s,x\rangle \in f_{r*}L_r$.
\item We have $\langle c_M(x),s\rangle \in f_{r*}L_r$.
\end{enumerate}
\item Let $r\in (0,r(K,\partial))$. For $M\in\C,x\in M$ and $s\in DM$, \TFAE.
\begin{enumerate}
\item We have $\langle s,x\rangle \in f_{r+*}L_{r+}$.
\item We have $\langle c_M(x),s\rangle \in f_{r+*}L_{r+}$.
\end{enumerate}
\end{enumerate}
\end{lemma}

\begin{proof}
As in the proof of Lemma \ref{ST2}, we can prove the assertion as an application of Lemma \ref{switch pi}.
\end{proof}

\begin{corollary}\label{description r}
For $r\in (0,r(K,\partial)]$ and $M\in \C$, $\Fr M=\{x\in M;\forall s\in DM,s(x)\in f_{r*}L_{r}\}$.
\end{corollary}


\subsection{Calculation of $F_{\pi}$}\label{pi}

Let $\pi$ be as in Example 1 in \S \ref{ex1}. We define the ultrametric function $|\ |_{\pi}:f_{\rK,bd*}N_{\pi}\to\mathbb{R}_{\ge 0};(a_i)\mapsto \sup\{|a_i|\pi_i;i\in\N\}$. The map defines the metric topology on the underlying set of $f_{\rK,bd*}N_{\pi}$. Recall that $f_{\rK,bd*}N_{\pi}$ as a topological space is complete (\cite[2.3.3/4]{bgr}). We define the ultrametric function $||\ ||_{\pi}:\KT\to\mathbb{R}_{\ge 0};(q_i)\mapsto \sup\{|i!q_i|\pi_0/\pi_i;i\in\N\}$. The map defines the metric topology on the underlying set of $\KT$.

\begin{lemma}\label{norm property}
\begin{enumerate}
\item For $(q_i)\in\KT,(a_i)\in f_{0*}L_0$, $(q_i)\cdot (a_i)=(\sum_{i=h+k-j,k\ge j}\binom{k}{j}\frac{1}{h!}\partial^h(j!q_j)a_k)_i$.
\item For $h,i,j,k\in\N$ such that $i=h+k-j,k\ge j$, $|\binom{k}{j}\frac{1}{h!}\partial^h(j!q_j)a_k|\pi_i\le |j!q_j|(\pi_0/\pi_j)|a_k|\pi_k\le |j!q_j|(\pi_0/\pi_j)|(a_i)|_{\pi}\le ||(q_i)||_{\pi}|(a_i)|_{\pi}$.
\item For $P\in \KT$, $||P||_{\pi}=0$ if and only if $P=0$
\item For $P,Q\in \KT$, $||P\pm Q||_{\pi}\le \max\{||P||_{\pi},||Q||_{\pi}\}$.
\item For $P\in \KT$, $||P||_{\pi}=sup \{|P\cdot x|_{\pi}/|x|_{\pi};x\in f_{\rK,bd*}N_{\pi},x\neq 0\}$.
\item For $P,Q\in \KT$, $||P\cdot Q||_{\pi}\le ||P||_{\pi}\cdot ||Q||_{\pi}$.
\end{enumerate}
\end{lemma}

\begin{proof}
1. We have $(q_i)\cdot (a_i)=\sum_{j=0}^{+\infty}q_jT^0\cdot (j!\binom{i+j}{j}a_{i+j})_i=\sum_{j=0}^{+\infty}(\sum_{i=h-j+k,k\ge j}\frac{1}{h!}\partial^h(q_j)\cdot j!\binom{k}{j}a_k)_i=(\sum_{i=h-j+k,k\ge j}\frac{1}{h!}\partial^h(q_j)\cdot j!\binom{k}{j}a_k)_i=(\sum_{i=h-j+k,k\ge j}\binom{k}{j}\frac{1}{h!}\partial^h(j!q_j)a_k)_i$.

2. By Lemma \ref{kx} and \ref{inequality}, $|\binom{k}{j}\frac{1}{h!}\partial^h(j!q_j)a_k|\pi_i\le |\binom{k}{j}||\frac{1}{h!}\partial^h(j!q_j)||a_k|\pi_i\le |\frac{1}{h!}\partial^h(j!q_j)||a_k|\pi_i\le |j!q_j||a_k|\pi_i/r(K,\partial)^h=|j!q_j|(\pi_0/\pi_j)|a_k|\pi_k(\pi_i\pi_j)/(\pi_0\pi_k r(K,\partial)^h)\le |j!q_j|(\pi_0/\pi_j)|a_k|\pi_k$. The rest of assertion is obvious.

3,4. It is obvious.

5. Fix $(q_i)_i\in \KT$. For any $(a_i)\in f_{0*}L_0$, we have $|(q_i)\cdot (a_i)|_{\pi}\le ||(q_i)||_{\pi}|(a_i)|_{\pi}$ by parts $1$ and $2$. Let $j_0$ denote the minimum $j$ such that $||(q_i)||_{\pi}=|j!q_j|\pi_0/\pi_j$. Let $(a_i)$ denote the $j_0$-th fundamental vector in $f_{0*}L_0$. Then $||(q_i)||_{\pi}|(a_i)|_{\pi}=|j_0!q_{j_0}|\pi_0$. By part 3, $(q_i)\in\KT,(a_i)\in f_{0*}L_0$, $(q_i)\cdot (a_i)=(\sum_{i=h+j_0-j,j_0\ge j}\binom{j_0}{j}\frac{1}{h!}\partial^h(j!q_j)a_{j_0})_i$. By part 1, $|(\sum_{i=h+j_0-j,j_0>j}\binom{j_0}{j}\frac{1}{h!}\partial^h(j!q_j)a_{j_0})_i|_{\pi}<|j_0!q_{j_0}|\pi_0$. By $|(\binom{j_0}{j_0}\frac{1}{i!}\partial^i(j!q_j)a_{j_0})_i|_{\pi}\ge |\binom{j_0}{j_0}\frac{1}{0!}\partial^0(j!q_j)a_{j_0}|\pi_0=|j_0!q_{j_0}|\pi_0$, $|(q_i)\cdot (a_i)|_{\pi}\ge ||(q_i)||_{\pi}|(a_i)|_{\pi}=|j_0!q_{j_0}|\pi_0$.

6. We may assume there exists $x\in f_{\rK,bd*}N_{\pi}$ such that $Q\cdot x\neq 0$; otherwise, $||P\cdot Q||_{\pi}=0$ by part 5. By part 5, $||P\cdot Q||_{\pi}\le sup \{|(P\cdot Q)\cdot x|_{\pi}/|x|_{\pi};x\in N_{\pi},x\neq 0\}=sup \{|(P\cdot Q)\cdot x|_{\pi}/|Q\cdot x|_{\pi};x\in f_{\rK,bd*}N_{\pi},Q\cdot x\neq 0\}sup \{|Q\cdot x|_{\pi}/|x|_{\pi};x\in f_{\rK,bd*}N_{\pi},Q\cdot x\neq 0\}\le sup \{|P\cdot x|_{\pi}/|x|_{\pi};x\in f_{\rK,bd*}N_{\pi},x\neq 0\} sup \{|Q\cdot x|_{\pi}/|x|_{\pi};x\in f_{\rK,bd*}N_{\pi},x\neq 0\}=||P||_{\pi}||Q||_{\pi}$.
\end{proof}

\begin{corollary}\label{ideal}
Let $I$ be a left ideal of $\KT$. The closure $cl_{\pi}I$ of $I$ with respcet to the topology as above is a left ideal of $\KT$.
\end{corollary}

\begin{proof}
We choose a generator $P$ of $I$. Let $Q,R\in cl_{\pi}I$. Then there exist sequences $(Q_i),(R_i)$ converging to $Q,R$ respectively. Then the sequences $(Q_i-R_i),(Q_iR_i)$ converge to $Q-R,QR$ by parts 4,5 of Lemma \ref{norm property}.
\end{proof}

\begin{lemma}\label{norm}
Let $I$ be a left ideal of $\KT$. We define the ultrametric function $||\ ||_{\pi,I}:\KT/I\to \R_{\ge 0}$ by $||P+I||_{\pi,I}=\inf \{||P+R||_{\pi};R\in I\}$. For $P\in \KT,x\in \KT/I$, $||P\cdot x||_{\pi,I}\le ||P||_{\pi}|x|_{\pi,I}$.
\end{lemma}

\begin{proof}
By Lemma \ref{norm property}, we have $||P\cdot x||_{\pi,I}=\inf \{||R||_{\pi};R\in\KT,R+I=P\cdot x\}\le \inf \{||P\cdot Q||_{\pi};Q\in\KT,Q+I=x\}\le \inf \{||P||_{\pi}||Q||_{\pi};Q\in\KT,Q+I=x\}=||P||_{\pi}||x||_{\pi,I}$.
\end{proof}

\begin{lemma}\label{BGR}
Let $I$ be a non-zero left ideal of $\KT$, $||\ ||_{\pi,I}:\KT/I\to\R_{\ge 0}$ as above. Assume $cl_{\pi}I=I$. For an arbitrary morphism $\chi:i_*(\KT/I)\to K$ in $Mod(K)$, there exists $C$ such that $|\chi(x)|\le C||x||_{\pi,I}$ for $x\in \KT/I$.
\end{lemma}
\begin{proof}
We regard $||\ ||_{\pi},||\ ||_{\pi,I}$ as ultrametric functions on $i_*\KT,i_*(\KT/I)$ respectively. Then $||\ ||_{\pi}$ is a norm on $i_*\KT$ by definition and $||\ ||_{\pi,I}$ is a semi-norm on $i_*(\KT/I)$ as in the previous lemma. 
Since $K$ is complete, $K$ is weakly cartesian (\cite{bgr}). Since $i_*(\KT/I)$ is of finite dimension, $||\ ||_{\pi,I}$ is equivalent to any norm on $i_*(\KT/I)$. We choose a basis $e_1,\dots,e_m$ of $i_*(\KT/I)$, and define the norm $|\ |'$ on $i_*(\KT/I)$ by $|c_1e_1+\dots+c_me_m|'=max_i|c_i|$. Then there exists $C'\in \R$ such that $|\ |'\le C'||\ ||_{\pi,I}$. We set $C:=C' max |\chi(e_i)|$. For $x\in i_*(\KT/I)$, we write $x=c_1e_1+\dots+c_me_m$. Then $|\chi(x)|=|\chi(c_1e_1+\dots+c_me_m)|=|c_1\chi(e_1)+\dots+c_m\chi(e_m)|\le max_i|c_i||\chi(e_i)|\le \max\{|c_i|\}\max\{|\chi(e_i)|\}=|x|' max\{|\chi(e_i)|\}\le C'||x||_{\pi,I}\max\{|\chi(e_i)|\}\le C||x||_{\pi,I}$.
\end{proof}

\begin{lemma}\label{closed}
Let $I$ be a left ideal of $\KT$. Assume $cl_{\pi}I=I$. Then $dim \H(\KT/I,f_{\rK,bd*}N_{\pi})=dim \KT/I$.
\end{lemma}

\begin{proof}
We set $n=dim \KT/I$. 
We prove that for $s\in\H(\KT/I,f_{0*}L_0)$ and $x\in\KT/I$, $s(x)\in f_{\rK,bd*}N_{\pi}$. We apply Lemma \ref{BGR} to the morphism $\chi:=\Phi_{\KT/I}(s)\in \H(i_*(\KT/I),K)$ in $Mod(K)$. Then there exists $C$ such that $|\chi(x)|\le C|x|_{\pi,I}$ for $x\in \KT/I$. Let $x\in \KT/I$. We have $s(x)=\Psi_{\KT/I}(\chi)(x)=(\chi(\frac{T^i}{i!}\cdot x))_i$. For $i\in \N$, $|\chi(\frac{1}{i!}T^i\cdot x)|\pi_i\le C ||\frac{1}{i!}T^i\cdot x||_{\pi,I}\pi_i\le C ||\frac{1}{i!}T^i||_{\pi}|x|_{\pi,I}\pi_i=C (\pi_0/\pi_i) |x|_{\pi,I}\pi_i=C \pi_0 |x|_{\pi,I}$ by Lemma \ref{norm property}. Hence $sup_i |\chi(\frac{T^i}{i!}\cdot x)|\pi_i\le C \pi_0 |x|_{\pi,I}<+\infty$. Therefore $s(x)\in f_{r*}N_{\pi}$. Hence the obvious injection $\H(\KT/I,f_{\rK,bd*}N_{\pi})\to \H(\KT/I,f_{0*}L_0)$ is a surjection. Therefore $dim \H(\KT/I,f_{\rK,bd*}N_{\pi})=dim \H(\KT/I,f_{0*}L_0)=dim \KT/I$ by Lemma \ref{duality}.
\end{proof}

\begin{proposition}\label{Fpi}
Let $I$ be a non-zero left ideal of $\KT$, $cl_{\pi}I$ as in Corollary \ref{ideal}. We regard $F_{\pi}(\KT/I),cl_{\pi}I/I$ as subobjects of $\KT/I$ in an obvious way. Then $F_{\pi}(\KT/I)=cl_{\pi}I/I$ as subobjects of $\KT/I$.
\end{proposition}
\begin{proof}
We prove $cl_{\pi}I/I\subset F_{\pi}(\KT/I)$...(1). Let $s\in \H(\KT/I,f_{\rK,bd*}N_{\pi})$. Let $s':\KT\to f_{\rK,bd*}N_{\pi}$ be the composition of the canonical surjection $\KT\to\KT/I$ followed by $s$. Then $|s'(P)|_{\pi}=|P\cdot s'(1)|_{\pi}\le ||P||_{\pi}|s'(1)|_{\pi}$ by Lemma \ref{norm property}. We endow the underlying sets of $\KT,f_{\rK,bd*}N_{\pi}$ with the topology associated to the norm $||\ ||_{\pi},|\ |_{\pi}$ respectively. Then $s'$ is continuous. Since $f_{\rK,bd*}N_{\pi}$ is Hausdorff and $s'(I)=0$, $s'(cl_{\pi}I)=0$. Hence $s(cl_{\pi}I/I)=0$. Therefore $cl_{\pi}I/I\subset F_{\pi}(\KT/I)$ by Proposition \ref{dimension formula}.

By Lemma \ref{closed} and Proposition \ref{dimension formula}, we have $dim F_{\pi}(\KT/I)=dim \KT/I-dim \H(\KT/I,f_{\rK,bd*}N_{\pi})\le dim \KT/I-dim \H(\KT/cl_{\pi}I,f_{\rK,bd*}N_{\pi})=dim \KT/I-dim \KT/cl_{\pi}I=dim cl_{\pi}I/I$. Together with (1), we obtain the assertion.
\end{proof}



\subsection{An analogue of Dwork's transfer theorem}

\begin{lemma}\label{simple}
Let $r\in (0,+\infty)$. Let $I$ be any principal ideal of $\KXrz$ such that $\partial(I) \subset I$. Then either $I=0$ or $I=\KXrz$. A similar assertion holds when we replace $\KXrz$ by $\KXr$ or $\KXrp$.
\end{lemma}

\begin{proof}
An analogous property for $\Tar$ is known, where $\Tar$ denotes the subring of $\KX$ consisting of $(a_i)\in K_+^{\N}$ such that $|a_i|r^i\to 0\ (i\to+\infty)$ equipped with the derivation induced by $\partial_X$ (\cite[Lemma 9.11]{pde}). In the case of $R=\KXr$, let $x$ denote a generator of $I$. Since $\KXr=\cap_{t\in (0,r)}\Tat$ and $\partial_X(\Tat\cdot x)\subset \Tat\cdot x$ for $t\in (0,r)$, we have $x\in \cap_{t\in (0,r)}(\Tat^{\times})=\KXr^{\times}$. In the case of $\KXrz$, by applying the previous case to the ideal $\KXr\cdot x$, $x\in \KXr^{\times}=K[[X/r]]_0^{\times}$. In the case of $\KXrp$, we choose $t\in (r,+\infty)$ such that $x\in\KXt$ and $\partial_X(x)\in \KXt$. Then $\partial_X(\KXt\cdot x)\subset \KXt\cdot x$, which implies $x\in \KXt^{\times}\subset\KXrp^{\times}$.
\end{proof}

We recall results of \cite[6.6]{chr} in the following special case. Let $r\in (0,\rK],m\in\N_{>0}$. We define the ring homomoprhism $\kappa_{r,bd}:K\to\KXrz;c\mapsto cX^0$. For $G=(g_{ij})\in M_m(\KXrz)$, we define the complex $C_{r,bd}(G)$ in $Mod(K)$ concentrated at degree $0$ and $1$ by 
$$
\xymatrix{
\dots\ar[r]&0\ar[r]&(\kappa_{r,bd*}\KXrz)^m\ar[r]^{(\partial_X-G)\cdot}&(\kappa_{r,bd*}\KXrz)^m\ar[r]&0\ar[r]&\dots,
}
$$
where $(\partial_X-G)\cdot (x_i)=(\partial_X(x_i)-\sum_{j=1}^mg_{ij}x_j)$; note that $\partial_X$ defines an endomorphism on $\kappa_{r,bd*}\KXrz$. We give similar definitions for $\KXr,\KXrp$.

\begin{lemma}\label{robba calculation}
Let notation be as above.

\begin{enumerate}
\item Let $G',G\in M_m(\KXrz),U\in GL_m(\KXrz)$ such that $\partial_X(U)+UG=G'U$. Then $C_{r,bd}(G)\cong C_{r,bd}(G')$. 
\item Let $0_m$ denote the zero matrix in $M_m(\KXrz)$. Then $dim H^0(C_{r,bd}(0_m))=m$ and $dim H^1(C_{r,bd}(0_m))=+\infty$ if $p(K)>0$ and $dim H^1(C_{r,bd}(0_m))=0$ if $p(K)=0$.
\item Let $M\in\Cf,G_1$ denote the matrix of the action of $T$ on $M$ for a given basis. Then there exists an isomorphism $i_*\Ex^j(M,f_{r,bd*}L_{r,bd})\cong H^j(C_{r,bd}(g_{r,bd}(G_1))$ in $Mod(K)$ for $j=0,1$.

Similar assertion for $\KXr,\KXrp$ hold except that in part 2, the result should be replaced by $dim H^0(C_{r}(0_m))=dim H^0(C_{r+}(0_m))=m$ and $dim H^1(C_r(0_m))=dim H^1(C_{r+}(0_m))=0$ regardless of $p(K)$.
\end{enumerate}
\end{lemma}

\begin{proof}
1. The morphism $(\kappa_{r,bd*}\KXrz)^m\to (\kappa_{r,bd*}\KXrz)^m$ defined by the left multiplication by $U$ induces the desired isomorphism.

2. The assertion for $H^0$ is obvious. The assertion for $H^1$ is due to \cite[Proposition 15.1]{chr} in the case of $\KXrz$ and obvious in the other cases.

3. We recall the construction of \cite[6.6]{chr}. We consider the projective resolution of $M$ given by
$$
\xymatrix{
\dots&\ar[r]&0\ar[r]&\tilde{\KT^m}\ar[r]^{\cdot (T-GT^0)}&\tilde{\KT^m}\ar[r]&M,
}
$$
where $\tilde{\KT^m}$ denote the left $\KT$-module given by the row vectors of length $m$ with entries in $\KT$, $\cdot (T-GT^0)$ denotes the right multiplication by the matrix $T-GT^0$, i.e., $(P_1,\dots,P_m)\cdot (T-GT^0)=(P_1\cdot T-P_1g_{11}-P_2g_{21}-\dots,\dots)$ and the last morphism sends the $j$-th fundamental vector $e_j$ of $\tilde{\KT^m}$ to the $j$-th basis of $M$. By applying the endofunctor $\H(-,f_{r,bd*}L_{r,bd})$ on $\C$, we obtain the complex $C$ in $\C$ concentrated at degree $0$ and $1$ $C:\dots\to 0\to \H(\tilde{\KT^m},f_{r,bd*}L_{r,bd})\to \H(\tilde{\KT^m},f_{r,bd*}L_{r,bd})\to 0\to \dots$. We have the isomorphism $\alpha:\H(\tilde{\KT^m},f_{r,bd*}L_{r,bd})_+\to (\KXrz^m)_+;s\mapsto (s(e_j))$. By the definition of $\kappa_{r,bd}$, this isomorphism extends to an isomorphism $i_*\H(\tilde{\KT^m},f_{r,bd*}L_{r,bd})\to (\kappa_{r,bd*}\KXrz)^m$. We can see that as a complex in $\A$, $C$ is isomorphic to $C_{r,bd}(g_{r,bd}(G_1))$ via $\alpha$. Hence the complex $i_*C$ obtained by applying $i_*$ to each term of $C$ is isomorphic to $C_{r,bd}(g_{r,bd}(G_1))$. Since the bifunctor $\Ex^i$ is naturally isomorphic to the $i$-th partial derived functor of $\H$ with only the first variable active (\cite[Theorem 8.1]{ce}), $i_*\Ex^1(M,f_{r,bd*}L_{r,bd})$ is isomorphic to $H^1(C)$ in $\C$, which implies the assertion.
\end{proof}

\begin{proposition}\label{trivial}
Let $M\in \Cf,r\in (0,\rK]$.

\begin{enumerate}
\item There exists an isomorphism $f_{r,bd}^*\H(M,f_{r,bd*}L_{r,bd})\to L_{r,bd}^{dim \H(M,f_{r,bd*}L_{r,bd})}$ in $\CKXrzT$.

\item \TTFAE
\begin{enumerate}
\item We have $dim \H(M,f_{r,bd*}L_{r,bd})=dim M$.

\item There exists an isomorphism $f_{r,bd}^*M\to L_{r,bd}^{dim M}$ in $\CKXrzT$.

\end{enumerate}

\end{enumerate}

In both parts 2 and 3, a similar assertion holds when we replace $f_{r,bd},L_{r,bd},\CKXrzT$ by $f_{r},L_{r},\CKXrT$ respectively or $f_{r+},L_{r+},\CKXrpT$ respectively.
\end{proposition}

\begin{proof}
We give a proof for $\KXrz$. For the other cases, a similar proof works.

1. We set $m=dim M$. We may assume that $dim \H(M,f_{r,bd*}L_{r,bd})=m$ by Proposition \ref{dimension formula} after replacing $M$ by $M/F_{\pi(r)}M$. Hence the obvious morphism $\H(M,f_{r,bd*}L_{r,bd})\to \H(M,f_{0*}L_0)$ is an isomorphism by comparing dimension. Since $i_{0*}f_0^*\H(M,f_{0*}L_{0})$ is a finite $\KX$-module, it is known that there exists an isomorphism $\beta'_0:f_0^*\H(M,f_{0*}L_{0})\to L_0^m$ in $\CKXT$. Let $\beta_0:\H(M,f_{0*}L_{0})\to f_{0*}(L_0^m)$ denote the morphism corresponding to $\beta_0$ under the adjunction isomorphism. Let $\beta_0^{(i)}:\H(M,f_{0*}L_{0})\to f_{0*}L_{0}$ be the composition of $\beta$ followed by the $i$-th projection. Then there exists a unique $x_i\in M$ such that $\beta_0^{(i)}=c_M(x_i)$ as $c_M$ is an isomorphism. Hence $\beta_0$ coincides with the morphism $s\mapsto (c_M(x_1)(s),\dots,c_M(x_m)(s))$. We consider the composition of the obvious morphism $\H(M,f_{r,bd*}L_{r,bd})\to \H(M,f_{0*}L_0)$ followed by $\beta_0$. By Lem \ref{switch r}, the morphism factors through $f_{r,bd*}(L_{r,bd}^m)$. Let $\beta_{r,bd}:\H(M,f_{r,bd*}L_{r,bd})\to f_{r,bd*}(L_{r,bd}^m)$ denote the morphism obtained in this way. Let $\beta'_{r,bd}:f_{r,bd}^*\H(M,f_{r,bd*}L_{r,bd})\to L_{r,bd}^m$ denote the morphism corresponding to $\beta_{r,bd}$ under the adjunction isomorphism. We consider the matrix $G_{r,bd}$ of the action $T$ on $f_{r,bd}^*\H(M,f_{r,bd*}L_{r,bd})$ and the matrix $X_{r,bd}$ (resp. $X_0$) of $\beta'_{r,bd}$ (resp. $\beta'_0$) with respect to the following basis: we choose an arbitrary basis $e_1,\dots,e_m$ of $\H(M,f_{r,bd*}L_{r,bd})$ and consider the $1\otimes e_i$'s as a basis of $f_{r,bd}^*\H(M,f_{r,bd*}L_{r,bd})$ (resp. $f_{0}^*\H(M,f_{0*}L_{0})$); we consider the standard basis of $L_{r,bd}^m$ (resp. $L_{0}^m$). Then we have $\partial_X(X_{r,bd})=G_{r,bd}X_{r,bd}$ and $X_{r,bd}=X_0$. We have $det(X_0)\neq 0$ by assumption. Hence we have $det (X_{r,bd})\neq 0$. We have $\partial_X(det (X_{r,bd}))=tr (G_{r,bd}) det (X_{r,bd})$ by calculating the determinant of $\partial_X(X_{r,bd})X'_{r,bd}=G_{r,bd}X_{r,bd}X'_{r,bd}$, where $X'_{r,bd}$ denotes adjugate matrix of $X_{r,bd}$. Hence the ideal $\KXr\cdot det(X_{r,bd})$ of $\KXr$ is non-zero and satisfies $\partial_X(\KXr\cdot det(X_{r,bd}))\subset \KXr\cdot det(X_{r,bd})$. Therefore $det (X_{r,bd})\in (\KXr)^{\times}$ by Lemma \ref{simple}, which implies that $\beta_{r,bd}$ is an isomorphism.

2. Let $m=dim M$.

(a)$\Rightarrow$(b) Assume (a) holds. By assumption, the obvious morphism $\H(M,f_{r,bd*}L_{r,bd})\to DM$ is an isomorphism by comparing dimension. By part 1, $f_{r,bd}^*DM\cong f_{r,bd}^*\H(M,f_{r,bd*}L_{r,bd})\cong L_{r,bd}^m$. Hence $f_{r,bd}^*D_0M\cong L_{r,bd}^m$. Let $G_1$ be a matrix of the action of $T$ on $M$ for a fixed basis of $M$. The matrix of the action of $T$ on $D_0M$ is given by $g_{r,bd}(-{}^tG_1)$, where the matrix is taken with respect to the dual basis of the fixed basis of $M$. The existence of the above isomorphism implies that there exists a matrix $X\in GL_m(\KXrz)$ such that $\partial_X(X)=g_{r,bd}(-{}^tG_1))X$. Then $\partial({}^tX^{-1})=g_{r,bd}(G_1){}^tX^{-1}$. Hence ${}^tX^{-1}$ defines an isomorphism $f_{r,bd}^*M\to L_{r,bd}^m$.

(b)$\Rightarrow$(a) Assume (b) holds. Let notation be as in Lemma \ref{robba calculation}. By assumption, there exists a matrix $U\in GL_n(\KXrz)$ such that $\partial(U)=g_{r,bd}(G_1) U$. As before, we have $dim \H(M,f_{r,bd*}L_{r,bd})=dim \Ex^0(M,f_{r,bd*}L_{r,bd})=dim H^0(C_{r,bd}(G_1))=dim H^0(C_{r,bd}(0_{dim M}))=m$.
\end{proof}

\begin{proposition}\label{DTT}
For $r\in (0,r(K,\partial)]$ and $M\in \Cf$, \TFAE
\begin{enumerate}
\item We have $dim \H(M,f_{r*}L_r)=dim M$.
\item There exists an isomorphism $f_r^*M\to L_r^{dim M}$ in $\CKXrT$.
\item We have $\supp m(M)\subset [r,r(K,\partial)]$.
\item The obvious morphism $\H(M,f_{r*}L_r)\to \H(M,f_{0*}L_0)$ in $\C$ is an isomorphism.
\item Let $e_1,\dots,e_m$ be any elements of $M$ which forms a basis of $i_*M$. Let $(G_k)$ denote the family of matrices where $G_k$ for $k\in\N$ denotes the matrix of the multiplication by $T^k$ on $M$ with respect to the $e_i$'s. Then $\omega/limsup_{i\in\N_{>0}}|G_i|^{1/i}\ge r$.
\item We have $F_{(0,r)}M=0$.
\end{enumerate}
\end{proposition}

\begin{proof}
The equivalences 2$\Leftrightarrow$3 and 3$\Leftrightarrow$5 are due to \cite[Proposition 1.2.14]{kx} and \cite[Lemma 6.2.5]{pde} respectively.

The equivalence 1$\Leftrightarrow$2 is a part of Proposition \ref{trivial}.

The equivalence 1$\Leftrightarrow$4 follows from $dim DM=dim M$.

The equivalence 1$\Leftrightarrow$6 follows by Proposition \ref{dimension formula}.
\end{proof}

\section{A decomposition theorem in $\Cf$ when $p(K)>0$}

We prove a decomposition theorem in $\Cf$ when $p(K)>0$, which is regarded as a stronger form of Theorem \ref{main theorem}. Besides an analogue of Dwork transfer theorem, the point is to prove that $f_{r*}L_r$ for $r\in (0,\rK]$ is an injective object in $\C$, where we use calculations of Christol developed in \cite{chr}. 



\subsection{The injectivities of $f_{r*}L_r$ and $f_{r+*}L_{r+}$ when $p(K)>0$}

\begin{lemma}\label{solvable lemma}
Let $r\in (0,\rK]$. Assume $F_{(0,r)}M=0$. Then $\Ex^1(M,f_{r*}L_r)=0$.
\end{lemma}

\begin{proof}
Let $G_1$ denote the matrix of the action of $T$ on $M$ with respect to a given basis. By Propositions \ref{dimension formula} and \ref{DTT}, there exists an isomorphism $f^*_rM\cong L_r^m$ with $m=dim M$. This implies that there exists $U\in GL_m(\KXr)$ such that $\partial(U)=f_r(G_1)U$. By Lemma \ref{robba calculation}, we have $i_*\Ex^1(M,f_{r*}L_r)\cong H^1(C_r(f_r(G_1)))\cong H^1(C_r(0_m))=0$.
\end{proof}

\begin{corollary}\label{solvable}
Let $r\in (0,\rK]$. $\Ex^1(M/F_{(0,r)}M,f_{r*}L_r)=0,\Ex^0(F_{(0,r)}M,f_{r*}L_r)=0,\H(F_{(0,r)}M,f_{r*}L_r)=0$.
\end{corollary}
\begin{proof}
The assertion follows from Proposition \ref{dimension formula} and Lemma \ref{solvable lemma}.
\end{proof}

\begin{lemma}\label{key lemma}
Assume $p(K)>0$. Let $r\in (0,\rK]$. \TFAE

(a) We have $\H(M,f_{r,bd*}L_{r,bd})=0$.

(b) We have $\Ex^1(M,f_{r,bd*}L_{r,bd})=0$.


(c) We have $F_{\pi(r)}M=M$


\end{lemma}
\begin{proof}
The equivalence between conditions (a) and (c) proved in Proposition \ref{DTT}.

Assume (a) does not hold. We have $dim \Ex^0(F_{\pi(r)}M,f_{r,bd*}L_{r,bd})=dim \H(F_{\pi(r)}M,f_{r,bd*}L_{r,bd})\le dim F_{\pi(r)}M<+\infty$. We set $n=dim M/F_{\pi(r)}M$. By assumption and Proposition \ref{dimension formula}, $n\ge 1$ and $F_{\pi(r)}(M/F_{\pi(r)}M)=0$. By Proposition \ref{DTT}, $f_{r,bd}^*(M/F_{\pi(r)}M)\cong L_{r,bd}^n$. By a similar argument as in the proof of Lemma \ref{solvable lemma}, $dim \Ex^1(M/F_{\pi(r)}M,f_{r,bd*}L_{r,bd})=dim H^1(C_{r,bd}(0_n))=+\infty$. Hence $dim \Ex^1(M,f_{r,bd*}L_{r,bd})=+\infty$.

(a)$\Leftrightarrow$(b)

Assume (a) holds. We may assume $M=\KT/I$ with $I$ a non-zero left ideal of $\KT$. By assumption and Proposition \ref{dimension formula}, $dim F_{\pi(r)}M=dim M$. We fix $P\in \KT$ such that $I=\KT\cdot P$. By $cl_{\pi(r)}I=\KT$, there exists $Q\in \KT$ such that $||Q\cdot P-1||_{\pi(r)}<1$. Note that $Q\neq 0$. We define $M',M''\in C$ by $M'=\KT/\KT\cdot (Q\cdot P),M''=\KT\cdot P/\KT\cdot (Q\cdot P)$. Then $M',M''\in \Cf$ and we have an isomorphism $\KT/\KT\cdot Q\cong M'$ induced by $\KT\to M';1\mapsto P$. Moreover, we have an exact sequence $E:0\to M''\to M'\to M\to 0$. For $x\in f_{r,bd*}L_{r,bd}$, $\sum_{n=0}^{+\infty}(1-Q\cdot P)^n\cdot x$ converges in $f_{r,bd*}L_{r,bd}$ with respect to the topology defined by $|\ |_{\pi(r)}$ since $|(1-Q\cdot P)^n\cdot x|_{\pi(r)}\le ||(1-Q\cdot P)^n||^n_{\pi(r)}|x|_{\pi(r)}$ by Lemma \ref{norm property}. Hence $x=(Q\cdot P)\cdot \sum_{n=0}^{+\infty}(1-Q\cdot P)^n\cdot x$ for $x\in f_{r,bd*}L_{r,bd}$, which implies that the multiplication map $(Q\cdot P)\cdot:(f_{r,bd*}L_{r,bd})_+\to (f_{r,bd*}L_{r,bd})_+$ is surjective. Therefore we have $\Ex^1(M',f_{r,bd*}L_{r,bd})=0$. We apply the functor $\H(-,f_{r,bd*}L_{r,bd})$ to $E$. Then $\Ex^1(M'',f_{r,bd*}L_{r,bd})=0$. By (b)$\Rightarrow$(a), which is proved above, we have $\Ex^0(M'',f_{r,bd*}L_{r,bd})\cong \H(M'',f_{r,bd*}L_{r,bd})=0$. Hence $\Ex^1(M,f_{r,bd*}L_{r,bd})=0$.
\end{proof}

\begin{lemma}\label{non solvable}
Let $r\in (0,\rK]$. Assume $p(K)>0$ and $F_{(0,r)}M=M$. Then $\Ex^1(M,f_{r*}L_r)=0$.
\end{lemma}
\begin{proof}
We have $\H(M,f_{r*}L_{r})=0$ by assumption. By Lemma \ref{limit lemma}, there exists $r_0\in (0,r)$ such that $\H(M,f_{t,bd*}L_{t,bd})=0$ for $t\in [r_0,r)$. In the following, let $t\in [r_0,r)$. By Lemma \ref{key lemma}, $\Ex^1(M,f_{t,bd*}L_{t,bd})=0$ for $t\in [r_0,r)$. We fix an isomorphism $M\cong\KT/\KT\cdot P$ fro $P\in\KT$. Then, for $M'\in\C$, $\Ex^1(M,M')$ is isomorphic to the cokernel of the morphism $P\cdot:M'_+\to M'_+$. Since we have $\H(M,f_{t,bd*}L_{t,bd})=0$, the morphism $P\cdot:(f_{t,bd*}L_{t,bd})_+\to (f_{t,bd*}L_{t,bd})_+$ is an isomorphism. By taking limit, $P\cdot:(f_{r*}L_r)_+\to (f_{r*}L_r)_+$ is an isomorphism, which implies the assertion.
\end{proof}

\begin{proposition}\label{vanishing}
Let $r\in (0,\rK]$. Assume $p(K)>0$. For an arbitrary $M\in\Cf$, $\Ex^1(M,f_{r*}L_r)=0$.
\end{proposition}
\begin{proof}
Since we have $F_{(0,r)}F_{(0,r)}M=F_{(0,r)}M$ by Proposition \ref{dimension formula} and Corollary \ref{solvable}, $\Ex^1(F_{(0,r)}M,f_{r*}L_r)=0$ by Lemma \ref{non solvable}. Hence Corollary  \ref{solvable} implies the assertion. 
\end{proof}

\begin{theorem}\label{injectivity}
Let $r\in (0,r(K,\partial)]$. Assume $p(K)>0$. We regard $f_{r*}L_r$ as an object of $\C$ by forgetting the structure of the left $\KT$-object. Then $f_{r*}L_r$ is an injective object in $\C$.
\end{theorem}

\begin{proof}
For any left ideal $I$ of $\KT$, we have $\Ex^1(\KT/I,f_{r*}L_r)=0$ by Proposition \ref{vanishing} in the case of $I\neq 0$ and by the projectivity of $\KT$ in the case of $I=0$.
\end{proof}

\begin{corollary}\label{injectivity p}
Let $r\in (0,r(K,\partial))$. Assume $p(K)>0$. We regard $f_{r+*}L_{r+}$ as an object of $\C$ by forgetting the structure of the left $\KT$-object. Then $f_{r+*}L_{r+}$ is an injective object in $\C$.
\end{corollary}

\begin{proof}
It follows from the existence of an isomorphism $f_{r+}L_{r+}\cong colim\{f_{s*}L_s\}_{s\in (r,r(K,\partial)]}$ in $\C$.
\end{proof}



\subsection{Statement and proof of the decomposition theorem when $p(K)>0$}

Throughout this subsection, we assume $p(K)>0$.

\noindent Convention. In the rest of the paper except \S \ref{decJ}, for an endofunctor $F$ on $\C$ equipped with a natural transformation $F\to id_{\C}$ which is a pointwise monomorphism, we denote by $F^f$ the endofunctor on $\Cf$ induced by $F$. By definition, the functor $F^f$ is equipped with a natural transformation $F^f\to id_{\Cf}$ which is a pointwise monomorphism.

\begin{proposition}\label{exact functor}
\begin{enumerate}
\item For $r\in (0,r(K,\partial)]$, the functor $\H(-,f_{r*}L_r)$ is exact.
\item For $r\in (0,r(K,\partial))$, the functor $\H(-,f_{r+*}L_{r+})$ is exact.
\item For $r\in (0,r(K,\partial)]$, the functors $F^f_{(0,r)},F^f_{[r,r(K,\partial)]}$ are exact.
\item For $r\in (0,r(K,\partial))$, the functor $F^f_{(0,r]}$ is exact.
\end{enumerate}
\end{proposition}

\begin{proof}
Parts 1, 2 follow from Theorem \ref{injectivity}, Corollary \ref{injectivity p} respectively. Parts 3, 4 are consequences of parts 1, 2 and Lemma \ref{duality} and Proposition \ref{dimension formula}.
\end{proof}

\begin{corollary}\label{cDDT}
\begin{enumerate}
\item For $r\in (0,r(K,\partial)]$ and $M\in\Cf$, \TFAE
\begin{enumerate}
\item $\supp m(M)\subset (0,r)$.
\item $\H(M,f_{r*}L_r)=0$.
\end{enumerate}
\item For $r\in (0,r(K,\partial))$ and $M\in\Cf$, \TFAE
\begin{enumerate}
\item $\supp m(M)\subset (0,r]$.
\item $\H(M,f_{r+*}L_{r+})=0$.
\end{enumerate}

\end{enumerate}
\end{corollary}
\begin{proof}
1. The negation of (b), i.e., $\H(M,f_{r*}L_r)\neq 0$ is equivalent to the condition that there exists a Jordan-H\"older factor $M^f$ of $M$ such that $\H(M^f,f_{r*}L_r)\neq 0$ by Proposition \ref{exact functor}; note that if this is the case then $dim \H(M^f,f_{r*}L_r)=dim M^f$ by Proposition \ref{dimension formula} and \ref{DTT}. Hence the last condition is equivalent to the one that there exists a Jordan-Holder factor $M^f$ of $M$ such that $\supp m(M)\not\subset (0,r)$ by Proposition \ref{DTT}.

2. By part 1, condition (a) is equivalent to the condition that for all $t\in (r,r(K,\partial)]$, $\H(M,f_{t*}L_t)=0$, which is equivalent to condition (b) by Lemma \ref{limit lemma}.
\end{proof}

\begin{proposition}\label{support}
Let $M\in \Cf$ and $r\in (0,r(K,\partial)]$.
\begin{enumerate}
\item We have $\supp m(\H(M,f_{r*}L_r))\subset [r,r(K,\partial)]$.
\item We have $\supp m(\Fr M)\subset [r,r(K,\partial)]$.
\item We have $\supp m(\Fzrh M)\subset (0,r)$.
\item Assume $r\neq r(K,\partial)$. $\H(\Fzr M,f_{r+*}L_{r+})=0$.
\item Assume $r\neq r(K,\partial)$. $\supp m(\Fzr M)\subset (0,r]$.
\end{enumerate}
\end{proposition}

\begin{proof}
1. We may assume $dim \H(M,f_{r*}L_r)=dim M$ after replacing $M$ by $M/\Fzrh M$ by Proposition \ref{dimension formula}. Hence the assertion follows from Proposition \ref{DTT}.

2. The assertion follows from $\Fr M\cong \H(DM,f_{r*}L_r)$ and part 1.

3. The assertion follows from Corollary \ref{solvable} and \ref{cDDT}.

4. It follows from Propositions \ref{dimension formula} and \ref{exact functor}.

5. The assertion follows from Corollary \ref{cDDT} and part 4.

\end{proof}


\begin{lemma}\label{SDT lemma}
Let $r\in (0,r(K,\partial)]$.
\begin{enumerate}
\item For $r\neq r(K,\partial)$, we have $\Frr=\Fzr \Fr$.
\item The functor $F^f_{[r,r]}$ is exact.
\item For $M\in \Cf$, we have $dim F^f_{[r,r]}M\ge m(M)(r)$.
\end{enumerate}


\end{lemma}

\begin{proof}
1. Let $M\in \C$. We have $\Fzr \Fr M\subset \Fzr M\cap \Fr M$. We prove the converse. Let $x\in \Fzr M\cap \Fr M$. Let $s\in \H(\Fr M,f_{r+*}L_{r+})$. By Corollary \ref{injectivity p}, $s$ extends to a morphism $s^f:M\to f_{r+*}L_{r+}$. By $x\in \Fzr M$, we have $s^f(x)=0$. Hence $s(x)=0$. Therefore $x\in \Fzr\Fr M$.

2. It follows from part 1 and Proposition \ref{exact functor}.

3. By part 2 and the additivity of $m(-)$, we may assume $M$ is irreducible. We have $\supp m(M)=\{s\}$ for $s\in (0,r(K,\partial)]$. In the case of $s\neq r$, the assertion holds by $m(M)(r)=0$. In the case of $s=r$, we have $\H(M,f_{r+*}L_{r+})=0$ by Corollary \ref{cDDT}. Hence $\Fzr M=M$ by Proposition \ref{dimension formula}. We have $dim \Fr M=dim \H(DM,f_{r*}L_r)=dim DM=dim M$ by $\supp m(DM)=\supp m(M)=\{r\}$ and Propositions \ref{dimension formula} and \ref{DTT}. Hence $\Fr M=M$. Thus $F^f_{[r,r]}M=M$. Hence $dim F^f_{[r,r]}M=dim M=m(M)(r)$.
\end{proof}

\begin{theorem}\label{SDT}
\begin{enumerate}
\item For $r\in (0,r(K,\partial)]$, we have $\Frs=F^f_{[r,r]}$.
\item The natural transformation $I_{sp}:\oplus_{r\in (0,r(K,\partial)]} F_{sp,[r,r]}\to id_{\Cf}$ is a natural isomorphism.
\end{enumerate}

\end{theorem}

\begin{proof}
We have $\supp m(F^f_{[r,r]}M)\subset \supp m(F^f_{(0,r]}M)\cap \supp m(F^f_{[r,r(K,\partial)]}M)=(0,r]\cap [r,r(K,\partial)]=\{r\}$ by Proposition \ref{support}. Hence $F^f_{[r,r]}M\subset \Frs M$. By Lemma \ref{SDT lemma}, $dim M\ge\sum_{t\in (0,r(K,\partial)]}dim F_{sp,[t,t]}M\ge \sum_{t\in (0,r(K,\partial)]}dim F^f_{[t,t]}M\ge \sum_{t\in (0,r(K,\partial)]}m(M)(t)=dim M$, which implies that $I_{sp,M}$ is an isomorphism. The inequalities also imply $dim \Frs M=dim F^f_{[r,r]}M$.
\end{proof}

\begin{corollary}
For $r\in (0,r(K,\partial)]$, the functor $F^f_{[r,r]}$ is exact.
\end{corollary}

\section{A decomposition theorem in $\Cf$ when $p(K)=0$}

We prove a decomposition theorem in $\Cf$ when $p(K)=0$, which is regarded as a stronger form of Theorem \ref{main theorem}. The reader note that Theorem \ref{main theorem} is proved in \cite{good formalI}. Our result is to give a description of the direct summands appearing in the decomposition.

Throughout this section, we assume $p(K)=0$.

\begin{theorem}\label{SDTz}
The obvious natural transformation $I_{sp}:\oplus_{r\in (0,r(K,\partial)]}\Frs\to id_{\Cf}$ is a natural isomorphism.
\end{theorem}
\begin{proof}
See \cite[Proposition 1.6.3]{good formalI}
\end{proof}

\begin{corollary}\label{exact lemmaz}
For $r\in (0,r(K,\partial)]$, the functor $F^f_{sp,[r,r]}$ is exact.
\end{corollary}

\begin{corollary}\label{supportz}
\begin{enumerate}
\item For $r\in (0,r(K,\partial)]$ and $M\in\Cf$, \TFAE
\begin{enumerate}
\item We have $\supp m(M)\subset (0,r)$.
\item We have $\H(M,f_{r*}L_r)=0$.
\end{enumerate}
\item For $r\in (0,r(K,\partial))$ and $M\in\Cf$, \TFAE
\begin{enumerate}
\item We have $\supp m(M)\subset (0,r]$.
\item We have $\H(M,f_{r+*}L_{r+})=0$.
\end{enumerate}

\end{enumerate}





\end{corollary}

\begin{proof}
1. We may assume $M\neq 0$. Since $m(-)$ is additive and $\H(-,f_{r*}L_r)$ commutes with finite direct sum, we may assume $\# \supp m(M)=1$ by Cor. \ref{SDTz}.

(a)$\Rightarrow$(b) Assume $\H(M,f_{r*}L_r)\neq 0$. We have $dim \H(M,f_{r*}L_r)=dim \H(M/\Fzrh M,f_{r*}L_r)=dim M/\Fzrh M>0$ by Proposition \ref{dimension formula}. Hence $\supp m(M/\Fzrh M)\subset [r,r(K,\partial)]$ by Proposition \ref{DTT}. Therefore $\supp m(M)$ contains an element of $[r,r(K,\partial)]$.

(b)$\Rightarrow$(a) Assume $\supp m(M)\not\subset (0,r)$. We write $\supp m(M)=\{s\}$ for $s\in [r,\rK]$. By Proposition \ref{DTT}, $dim \H(M,f_{r*}L_r)=dim M>0$.

2. We can prove as in the proof of Corollary \ref{cDDT}.
\end{proof}

\begin{proposition}\label{supportz2}
Let $r\in (0,\rK]$ and $M\in \Cf$.

1. We have $\supp m(\H(M,f_{r*}L_r))\subset [r,r(K,\partial)]$.

2. We have $\supp m(\Fr M)\subset [r,r(K,\partial)]$.

3. We have $\supp m(\Fzrh M)\subset (0,r)$.
\end{proposition}

\begin{proof}
Parts 1,2 are proved by a similar way as in the proof of Proposition \ref{support}.

3. The assertion follows from Corollary \ref{solvable} and Proposition \ref{supportz}.
\end{proof}

\begin{theorem}\label{comparisonz}
For all $r\in (0,r(K,\partial)]$, we have $\Frs=F^f_{[r,r]}$.
\end{theorem}

\begin{proof}
We prove $\Frr M\subset \Frs M$. In the case of $r=\rK$, the assertion follows from Proposition \ref{supportz2}. We assume $r\neq\rK$. 
We have $\Fzr M\subset \cap_{t\in (r,r(K,\partial)]}F_{(0,t)}M$ as we have an obvious monomorphism $f_{t*}L_t\to f_{r+*}L_{r+}$ for $t\in (r,\rK]$. Hence we have $\supp m(\Frr M)\subset \supp m(\Fzr M)\cap \supp m(\Fr M)\subset \cap_{t\in (r,r(K,\partial)]} \supp m(F_{(0,t)} M)\cap \supp m(\Fr M)\subset \cap_{t\in (r,r(K,\partial)]} (0,t)\cap [r,r(K,\partial)]=\{r\}$ by Proposition \ref{supportz2}.

We prove $\Frs M\subset \Frr M$. By Proposition \ref{DTT}, the obvious morphism $\H(\Frs M,f_{r*}L_r)\to\H(\Frs M,f_{0*}L_0)$ is an isomorphism. Hence, for $s\in DM$, $s(\Frs M)\subset f_{r*}L_r$. Therefore $\Frs M\subset \Fr M$. Assume $r\neq \rK$. By Proposition \ref{supportz}, $\H(\Frs M,f_{r+*}L_{r+})=0$. Hence, for $s\in \H(M,f_{r+*}L_{r+})=0$, $s(\Frs M)=0$. Therefore $\Frs M\subset \Fzr M$. Thus $\Frs M\subset\Frr M$.
\end{proof}

\begin{corollary}\label{exactz}
The functor $\Fr$ for $r\in (0,r(K,\partial)]$ is exact.
\end{corollary}

\section{Base change property of the function $m(-)$ and a rationality result}

We prove two properties of the function $m(-)$. One is a base change property and another concerns the values $r^{m(M)(r)}$ for $M\in\Cf$. The results of this section is not used in the proof of Theorem \ref{multiple SDT}.



\subsection{A base change property of the function $m$}

Let $K'$ be a complete non-archimedean valuation field of characteristic $0$ with non-trivial valuation equipped with a bounded non-zero derivation, which is denoted by $\partial$ for simplicity. Let $\sigma:K\to K'$ be a ring homomorphism preserving the valuations and commuting the derivations. Note that we have $p(K')=p(K),\omega(K')=\omega(K),\rKp\le \rK$, where $r(K',\partial)$ defined as before. Let $i':K\to\KpT$ denote the ring homomorphism as before and $\tau:\KT\to \KpT$ the ring homomorphism corresponding to the data $(\KpT,i'\circ \sigma,T)$ in the sense of Lemma \ref{universality}. We repeat the basic constructions as before after replacing $K$ by $K'$. In this way, we obtain $g'_0:K\to K'[[X]],f'_0:K\to K'[[X]]\langle T\rangle,L'_r=L_r(K'\{X/r\},\partial),N'_{\pi}$ and so on. We denote the functors $\Frs,\Fr$ and so on for $K'$ by $\Frs,\Fr$ and so on for simplicity. Note that if $M\in \Cf$ then $\tau^*M\in Mod^f(\KpT)$ by Lemma \ref{base change}.

\begin{lemma}\label{bc inequality}
Let $M\in \Cf$.
\begin{enumerate}
\item  For $\pi$ satisfying condition (C) for $K'$, we have $dim \H(\tau^*M,f'_{\rKp,bd*}N'_{\pi})\le dim \H(M,f_{\rK,bd*}N_{\pi})$.
\item For $r\in (0,r(K',\partial)]$, $dim \H(\tau^*M,f'_{r*}L'_r)\le dim \H(M,f_{r*}L_r)$.
\end{enumerate}
\end{lemma}

\begin{proof}
1. We may assume $M$ is of the form $\KT/\KT\cdot P$ with $P\in \KT$. We have an isomorphism $\tau^*M\cong \KpT/\KpT\cdot P$. Let $Q,Q'$ denote the monic generator of $cl_{\pi}(\KT\cdot P),cl_{\pi}(\KpT\cdot \tau(P))$ respectively. Then $dim \H(M,f_{\rK,bd*}N_{\pi})=deg(Q),dim \H(\tau^*M,f'_{\rKp,bd*}N'_{\pi})=deg(Q')$ by the result of \S \ref{pi}, where $deg(X)$ for $X\in \KT$ denotes the degrees of $X$ regarded as the usual polynomial in $K[T]$. Since the ultrametric functions $||||_{\pi}$ on $\KT$ and $\KpT$ are invariant under $\tau:\KT\to \KpT$, $\tau(cl_{\pi}(\KT\cdot P))\subset cl_{\pi}\tau(\KT\cdot P)\subset cl_{\pi}(\KpT\cdot \tau(P))$. Hence $\tau(Q)\in \KpT\cdot Q'$. Therefore $deg(\tau(Q))=deg(Q)\ge deg(Q')$.

2. It follows from Lemma \ref{limit lemma} and part 1.
\end{proof}

\begin{lemma}\label{dual invariance}
Let $M\in\Cf$. For $r\in (0,\rK]$, we have $dim \H(DM,f_{r*}L_r)=dim \H(M,f_{r*}L_r)$.
\end{lemma}
\begin{proof}
By Theorem \ref{SDT} and Corollary \ref{SDTz} and Proposition \ref{dimension formula}, we have $dim \H(M,f_{r*}L_r)=dim \H(\oplus_{t\in (0,\rK]}F_{sp,[t,t]}M,f_{r*}L_r)=\sum_{t\in (0,\rK]}dim \H(F_{sp,[t,t]}M,f_{r*}L_r)=\sum_{t\in (r,\rK]}dim \H(F_{sp,[t,t]}M,f_{0*}L_0)=\sum_{t\in [r,r]}m(M)(t)$. Hence $dim \H(M,f_{r*}L_r)=\sum_{t\in [r,r]}m(M)(t)$. Since we also have a similar equality for $DM$, the assertion follows from $m(DM)=m(D)$.
\end{proof}

\begin{proposition}\label{bc invariance}
Let $M\in \Cf$. For $r\in (0,r(K',\partial)]$, $m(\tau^*M)(r)=m(M)(r)$ if $r\neq r(K',\partial)$ and $m(\tau^*M)(r)=\sum_{r\in [r(K',\partial),r(K,\partial)]}m(M)(r)$ if $r=r(K',\partial)$.
\end{proposition}
\begin{proof}
By Theorems \ref{SDT} and \ref{comparisonz} and Corollary \ref{SDTz}, we have $m(M)(r)=dim \Frs M=dim F_{sp,[r.\rK]} M-lim_{s\to r+0}dim F_{sp,[s,r(K,\partial)]}M =dim \Fr M-lim_{s\to r+0}dim F_{[s,r(K,\partial)]}M$ for $r\in (r,r(K',\partial))$ and $\sum_{r\in [r(K',\partial),\rK]}m(M)(r)=dim M-\sum_{r\in (0,r(K',\partial))}m(M)(r)$  for $r\in (r,r(K',\partial))$. Similarly, $m(\tau^*M)(r)=dim F_{[r,r(K',\partial)]}\tau^*M-lim_{s\to r+0}dim F_{[s,r(K',\partial)]}\tau^*M$ for $r\in (r,r(K',\partial))$, and $m(\tau^*M)(r(K',\partial))=dim M-\sum_{r\in (0,r(K',\partial))}m(\tau^*M)(r)$ for $r\in (r,\rKp)$. Hence we have only to prove  $dim \Fr M=dim F_{[r,r(K',\partial)]}\tau^*M$.

Let $\tau_r:\KXrT\to \KpXrT$ be the ring homomorphism associated to the ring homomorphism $\KXr\to \KpXr;(a_i)\mapsto (\sigma(a_i))$. By $f'_r\circ\tau=\tau_r\circ f_r$ and Proposition \ref{DTT} and Theorems \ref{SDT} and \ref{comparisonz}, $(f'_r)^*\tau^* \Fr M\cong\tau_r^*f_r^* \Fr M\cong\tau_r^*(L_r^n)\cong (\tau_r^*L_r)^n\cong (L'_r)^n$ with $n=dim \Fr M$. By Proposition \ref{DTT}, $\supp m(\tau^* \Fr M)\subset [r,\rKp]$. Hence $dim \Fr M=dim \tau^*\Fr M\le dim F_{sp,[r,r(K',\partial)]} \tau^*M=dim F_{[r,r(K',\partial)]} \tau^*M$ by Theorems \ref{SDT} and \ref{comparisonz}. By Proposition \ref{dimension formula} and Lemma \ref{bc inequality} and Lemma \ref{dual invariance}, $dim \Frp \tau^*M=dim \H(D\tau^*M,f'_{r*}L'_r)=dim \H(\tau^*M,f'_{r*}L'_r)\le dim \H(M,f_{r*}L_r)=dim \H(DM,f_{r*}L_r)=dim \Fr M$. Hence $dim \Fr M=dim F_{[r,r(K',\partial)]}\tau^*M$.

\end{proof}



\subsection{A rationality result}\label{bc}

In this subsection, we regard $\R_{>0}$ as a subgroup of $\R^{\times}$. Moreover we endow $\R_{>0}$ with the unique structure of $\Q$-module. For a non-empty subset $G$ of $\R_{>0}$, we denote $G^{\Z}$ (resp. $G^{\Q}$) the subgroup (resp. the $\Q$-submodule) of $\R_{>0}$ generated by $G$.


For $r\in (0,+\infty)$, we define the complete non-archimedean valuation field $K(X)_s$ of characteristic $0$ as the completion of the rational function field $K(X)$ with respect to the $s$-Gauss valuation $||_s$. By abuse of notation, let $\partial_X$ denote the derivation on $K(X)_s$ defined as the unique extension of the derivation $K[X]\to K[X];\sum_ia_iX^i\mapsto \sum_i(i+1)a_{i+1}X^i$. Note that $K(X)_s$ is of rational type in the sense of \cite{kx} and $r(K(X)_s,\partial_X)=s$ and $|K(X)_s^{\times}|_s=|K(X)^{\times}|_s=|K^{\times}|\{s\}^{\Z}$.

\begin{lemma}\label{embedding}
For $s\in (0,r(K,\partial))$, there exists a ring homomorphism $g_{s,\eta}:K\to K(X)_s$ satisfying the following conditions.

1. The ring homomorphism $g_{s,\eta}$ commutes with the $\partial$ and $\partial_X$.

2. The ring homomorphism $g_{s,\eta}$ is isometric.
\end{lemma}


\begin{proof}
For $c\in K,i\in\N$, we have $|\partial^i(c)/i!|s^i\le |c|s^i/r(K,\partial)^i=(s/r(K,\partial))^i|c|$ by Lemma \ref{kx}. By using this inequality, we can define the ring homomorphism $g_{s,\eta}:K\to K(X)_s$ by $g_{s,\eta}(c)=\sum_i\frac{\partial^i(c)}{i!}X^i$. Part 1 holds obviously and part 2 follows from Lemma \ref{kx} with $r=s$.
\end{proof}


\begin{proposition}\label{rationality}
Let $M\in \Cf$. If $r\in (0,r(K,\partial))$ satisfies $m(M)(r)\neq 0$, then $r\in |K^{\times}|^{\Q}$ if $p(K)>0$, $r\in |K^{\times}|$ if $p(K)=0$.
\end{proposition}

\begin{proof}
Let $s\in (r,r(K,\partial))$ be arbitrary. Let $g_{s,\eta}:K\to K(X)_s$ be a ring homomorphism as in Lemma \ref{embedding}. Let $f_{s,\eta}:\KT\to K(X)_s\langle T\rangle$ be the ring hom. associated to $g_{s,\eta}$. Let $\lambda$ denote $\Q$ when $p(K)>0$ and (empty) when $p(K)=0$. Then we have $(r/s)^{m(f_{s,\eta}^*M)(r)}\in |K(X)^{\times}_s|^{\lambda}=(|K^{\times}|\{s\}^{\Z})^{\lambda}$ by applying \cite[Thm.1.4.21]{kx} to the finite differential module $Vf_{s,\eta}^*M$ over $K(X)_s$. Hence $r^{m(M)(r)}=r^{m(f^*_{s,\eta}M)(r)}\in (|K^{\times}|\{s\}^{\Z})^{\lambda}$...(1) by Proposition \ref{bc invariance}. As the valuation of $K$ is non-trivial, the subgroup $|K^{\times}|$ of $\R_{>0}$ is either of finite free of rank $1$ or dense in $\R_{>0}$, where $\R_{>0}$ is regarded as a topological abelian group with respect to Euclidean topology.

Case 1: $|K^{\times}|$ is of finite free of rank $1$.

We choose $s_0\in |K^{\times}|$ such that $|K^{\times}|=\{s_0\}^{\Z}$. We choose $s_1\in (r,r(K,\partial))\setminus \{s_0\}^{\Q}$ and $s_2\in (r,r(K,\partial))\setminus \{s_0,s_1\}^{\Q}$. Then $s_0,s_1,s_2$ is linearly independent in $\R_{>0}$. By applying (1) to $s=s_1,s_2$, we have $r^{m(M)(r)}\in \{s_0\}^{\lambda}\{s_1\}^{\Q}\cap \{s_0\}^{\lambda}\{s_2\}^{\Q}=\{s_0\}^{\lambda}=|K^{\times}|^{\lambda}$.

Case 2: $|K^{\times}|$ is dense in $\R_{>0}$.

We choose $s\in (r,r(K,\partial))\cap |K^{\times}|$. By (1), we have $r^{m(M)(r)}\in |K^{\times}|^{\lambda}$.
\end{proof}

\section{A decomposition theorem in $\CJf$}\label{decJ}

In this section, for a differential field $(K,\partial_J)$ as in the introduction, we recall the definition of the ring $\KTJ$, which is an analogue of $\KT$, and recall basic facts on the category $\CJf$, including the fact that $\CJf$ is isomorphic to the category of differential modules over $(K,\partial_J)$. Then we prove a decomposition theorem in $\CJf$ corresponding to Theorem \ref{main theorem}. The point of the proof is that we can lift the previous decomposition theorem to $\CJf$ by using the idea of internal $\H$.



\subsection{Basic definition}

Let $R$ be a commutative ring equipped with a family $\{\partial_j\}_{j\in J}$ of commuting derivations indexed by a non-empty set $J$. We use the multi-index notation, for example, we denote by $\partial_J$ the family $\{\partial_j\}_{j\in J}$. A differential module over $R$ is an $R$-module $V$ equipped with a family of commuting differential operators relative to $\partial_j$ for $j\in J$.

We repeat a similar construction as in \S \ref{def} in the above setup. Let $\N^{(J)}$ denote the monoid of maps $J\to\N;j\mapsto n_j$ such that $n_j=0$ for all but finitely many $j\in J$. We denote the map $(j\mapsto n_j)$ by $n_J=(n_j)$. Let $0_J$ denote the element of $\N^{(J)}$ given by the map $(j\mapsto 0)$. Let $e_j$ denote the element $n_J$ of $\N^{(J)}$ defined by $n_j=1$ and $n_{j'}=0$ for $j\neq j'$. Let $G(R,\partial_J)$ be the differential module over $R$ given by the free $R$-module $R^{(\N^{(J)})}$ equipped with the differential operators relative to $\partial_j$ for $j\in J$ defined by $\partial_j((q_{n_J}))=(\partial_j(q_{n_J})+q_{n_J-e_j})$ for $j\in J$, where we set $q_{n_J-e_j}=0$ if $n_j=0$. We denote by $e$ the element of $G(R,\partial_J)$ given by $e(0_J)=1$ and $e(n_J)=0$ for $n_J\neq 0_J$. We denote by $T_j$ for $j\in J$ the element of $G(R,\partial_J)$ given by $T_j(e_j)=1$ and $T_j(n_J)=0$ for $n_J\neq e_j$. The functor $\H(G(R,\partial_J),-)$ represents the forgetful functor $()_+:\DJJ\to \A$. Precisely speaking, we have an isomorphism $\H(G(R,\partial_J),V)\to V_+;s\mapsto s(e)$ for $V\in \DJJ$, whose inverse is given by  $V_+\to \H(G(R,\partial_J),V);x\mapsto ((q_{n_J})\mapsto \sum_{n_J}q_{n_J}\partial^{n_J}(x))$. Thus $G(R,\partial_J)$ is a projective generator of $\DJJ$. We define the ring $\RTJ$ as the abelian group $R^{(\N^{(J)})}_+$ equipped with the unique multiplication such that the isomorphism $\H(G(R,\partial_J),G(R,\partial_J))\to R^{(\N^{(J)})}_+$ extends to an isomorphism $\H(G(R,\partial_J),G(R,\partial_J))^{op}\to \RTJ$ of rings. As in \S \ref{def}, the category $\CRTJ$ of left $\RTJ$-modules is isomorphic to the category of differential modules $\DJJ$ over $R$, where we can define the isomorphism in an obvious way. Since we can easily translate results on $\CRTJ$ in terms of $\DJJ$ as in \S \ref{def}, we will study $\CRTJ$. As in \S \ref{def}, by writing elements of $\RTJ$ as $\sum q_{n_J}T_J^{n_J}$, we can calculate the multiplication in $\RTJ$ by using the relation $(qT_j)\cdot (q'T_{j'})=q\partial_j(q')T_{j'}+qq'T_jT_{j'},T_j\cdot T_{j'}=T_{j'}\cdot T_{j}$ and $qT_J^{0_J}\cdot q'T_J^{0_J}=q'T_J^{0_J}\cdot qT_J^{0_J}=(qq')T_J^{0_J}$ for $j,j'\in J,q,q'\in R$.

Let $i_J:R\to \RTJ;q\mapsto qT_J^0$ be the ring homomorphism.


\begin{lemma}\label{multiple universality}
Let $R$ be as above. We consider the data $(U,\mu,\{u_j\}_{j\in J})$ where $U$ is a ring, $\mu:R\to U$ is a ring homomorphism, $u_j\in U$ for $j\in J$ such that $u_ju_{j'}=u_{j'}u_{j}$ for $j,j'\in J$ and $u_j\cdot\mu(r)=\mu(r)\cdot u_j+\mu (\partial_j(r))$ for $r\in R$ and $j\in J$. Then there exists a unique ring homomorphism $f:\RTJ\to U$ such that $\mu=f\circ i_J$ and $f(T_j)=u_j$ for $j\in J$. Moreover, we have $f((q_{n_J}))=\sum_{n_J}\mu(q_{n_J})u_J^{n_J}$ for $(q_{n_J})\in \RTJ$.

The ring homomorphism $f:\RTJ\to U$ is called the ring homomorphism corresponding to the data $(U,\mu,\{u_j\}_{j\in J})$.
\end{lemma}
\begin{proof}
A similar proof as in the proof of Lemma \ref{universality} works, where we define the object $MU'\in\DJJ$ by $\mu_*U$ equipped with the differential operators relative to $\partial_j$ for $j\in J$ given by the left multiplication by $u_j\in U$.
\end{proof}

Let $\xi:R\to \E(R_+)$ be the ring homomorphism given by the multiplication of $R$. We define the left $\RTJ$-module $L(R,\partial_J)$ by $R_+\in\A$ equipped with the ring homomorphism $\RTJ\to \E(R_+)$ corresponding to the data $(\E(R_+),\xi,\{\partial_j\}_{j\in J})$.



\subsection{The functor $\FrJs$}

We consider the case where $R$ is a complete non-archimedean valuation field $K$ of characteristic $0$ with non-trivial valuation $|\ |$. 
We assume that the derivations $\partial_j$ are non-zero and bounded, i.e., the action $\partial_j$ on $K$ has a finite operator norm. We define $p(K)$ and $\omega(K)$ as before. Let $r(K,\partial_j)$ denote the ratio $\omega(K)/|\partial_j|_{sp,K}$. We define the dimension $dim M$ of a left $\KTJ$-module $M$ as that of $i_{J*}M$. We say that $M$ is of finite dimension if $dim M<+\infty$. We denote the category of finite dimensional left $\KTJ$-modules by $\CJf$. The category is an abelian full subcategory of $\CJ$.

Let $(0,r(K,\partial_J)]$ denote the set $\Pi_{j\in J}(0,r(K,\partial_j)]\subset \R^J$. Let $r_J\in (0,r(K,\partial_J)]$. Let $M\in\CJf$. As in \S \ref{def}, there exists a maximum submodule $N$ of $M$ such that $\supp m(h_{j*}N)\subset \{r_j\}$ for all $j\in J$, which we denote by $\FrJs M$. 
For a morphism $\alpha:M'\to M$ in $\CJf$, we can define the morphism $\FrJs\alpha:\FrJs M'\to\FrJs M$ by $\alpha(\FrJs M')\subset \FrJs M$. Then $\FrJs$ is an endofunctor on $\CJf$, which is equipped with the obvious natural transformation $\FrJs\to id_{\CJf}$.
Note that in the case of $\# J>1$, the definition does not immediately implies that for $M\in\CJf$, we have $\FrJs M=0$ for all but finitely many $r_J\in (0,\rKJ]$.



\subsection{Statement and proof of the decomposition theorem}

\begin{definition}
\begin{enumerate}
\item Let $j\in J$. We endow $\KX$ with the family of derivations $\partial_{j,J}=\{\partial_{j,j'}\}_{j'\in J}$ defined by $\partial_{j,j}((a_i))=\partial_X((a_i))=((i+1)a_{i+1})$ and $\partial_{j,j'}((a_i))=(\partial_{j'}(a_i))$ for $j'\in J$ such that $j'\neq j$. Then $(\KX,\partial_{j,J})$ is a differential ring. We denote by $\KXTJ$ the ring of twisted polynomials associated to the differential ring $(\KX,\partial_{j,J})$. Let $L_{j,0}$ denote the left $\KXTJ$-module $L(\KX,\partial_{j,J})$. We define the ring homomorphism $g_{j,0}:K\to \KX$ by $x\mapsto (\partial_j^i(x)/i!)$. We define the ring homomorphism $f_{j,0}:\KTJ\to \KXTJ$ as the one corresponding to the data $(\KX,g_{j,0},\{T_j\}_{j\in J})$. We obtain the left $\KTJ$-modules $f_{j,0*}L_{j,0}$.

\item Let $j\in J$ and $r\in (0,r(K,\partial_j)]$ (resp. $r\in (0,r(K,\partial_j))$). We endow $\KXr$ (resp. $\KXrp$) with the family of derivations $\partial_{j,J}=\{\partial_{j,j'}\}_{j'\in J}$ defined by $\partial_{j,j}((a_i))=\partial_X((a_i))=((i+1)a_{i+1})$ and $\partial_{j,j'}((a_i))=(\partial_{j'}(a_i))$ for $j'\in J$ such that $j'\neq j$. Then $(\KXr,\partial_{j,J})$ is a differential ring. We denote by $\KXrTJ$ (resp. $\KXrpTJ$) the ring of twisted polynomials associated to the differential ring $(\KXr,\partial_{j,J})$ (resp. $(\KXrp,\partial_{j,J})$). Let $L_{j,r}$ (resp. $L_{j,r+}$) denote the left $\KXrTJ$-module $L(\KXr,\partial_{j,J})$ (resp. the left $\KXrpTJ$-module $L(\KXrp,\partial_{j,J})$). We define the ring homomorphism $g_{j,r}:K\to \KXr$ (resp. $g_{j,r+}:K\to \KXrp$) by $x\mapsto (\partial_j^i(x)/i!)$. We define the ring homomorphisms $f_{j,r+}:\KTJ\to \KXrpTJ,f_{j,r}:\KTJ\to \KXrTJ$  as the one corresponding to the data $(\KXr,g_{j,r},\{T_j\}_{j\in J})$ (resp. $(\KXrp,g_{j,r+},\{T_j\}_{j\in J})$). We obtain the left $\KTJ$-modules $f_{j,r+*}L_{j,r+},f_{j,r*}L_{j,r}$.

\item For $j\in J$, we denote the ring of twisted polynomials associated to $(K,\partial_j)$ by $\KTj$. For $r\in (0,+\infty)$, we repeat the construction as before for the differential ring $(\KX,\partial_X)$ (resp.$(\KXr,\partial_X),(\KXrp,\partial_X)$). We denote the ring of twisted polynomials associated to $(\KX,\partial_X)$ (resp. $(\KXr,\partial_X),(\KXrp,\partial_X)$) by $\KXTT$ (resp. $\KXrTT,\KXrpTT$).
 
\item Let $j\in J$ and $r\in (0,r(K,\partial_j)]$ (and, without mentioning, assume $r\in (0,r(K,\partial_j))$ when the construction involves $\KXrpTj$ as before). We define the ring homomorphisms $\f_{j,0}:\KTj\to \KXTT,\f_{j,r}:\KTj\to \KXrTj,\f_{j,r+}:\KTj\to \KXrpTj$ associated to $g_{j,0},g_{j,r},g_{j,r+}$ respectively by using the universality of $\KTj$. 

Let $j\in J$ and $r\in (0,r(K,\partial_j)]$. We define the ring homomorphism $h_j:\KTj\to\KTJ$ (resp. $h_{j,r}:\KXrTT\to\KXrTJ,h_{j,r+}:\KXrpTT\to\KXrpTJ$) as the one corresponding to the data $(\KTJ,i_J,T_j)$ (resp. $(\KXrTJ,\KXr\to \KXrTJ;x \mapsto xT^0_{j,j},T_{j,j}),(\KXrpTJ,\KXrp\to \KXrpTJ;x \mapsto xT^0_{j,j},T_{j,j})$).

Let $\L_0$ (resp. $\L_{r},\L_{r+}$) denote the left $\KXTT$-module $L(\KXTT,\partial_{X})$ (resp. the left $\KXrTj$-module $L(\KXr,\partial_{X})$, the left $\KXrpTj$-module $L(\KXr,\partial_{X})$). We obtain the left $\KTj$-modules $\f_{j,0}\L_0,\f_{j,r+*}\L_{r+},\f_{j,r*}\L_{r}$.

We have the following commutative diagram by construction. We have similar commutative diagrams when $\KXrTj$ is replaced by $\KXrpTj,\KXTT$.

$$
\xymatrix{
K\ar[r]^{i_j}\ar[d]^{=}&\KTj\ar[r]^{\f_{j,r}}\ar[d]^{h_j}&\KXrTj\ar[d]^{h_{j,r}}\\
K\ar[r]^{i_J}&\KTJ\ar[r]^{f_{j,r}}&\KXrTJ
}
$$
\end{enumerate}
\end{definition}

\begin{lemma}\label{internal}
Fix $j\in J$ and $M',M\in \CJ$. For $j'\in J,j'\neq j$ and $s\in \H(h_{j*}M',h_{j*}M)$, we define the morphism $\Delta_{j'}(s):M'_+\to M_+$ in $\A$ by $\Delta_{j'}(s)(x)=T_{j'}\cdot (s(x))-s(T_{j'}\cdot x)$ for $x\in M'$. Then $\Delta_{j'}(s)$ is a morphism $h_{j*}M'\to h_{j*}M$ in $\Cj$.
\end{lemma}

\begin{proof}
Since the ring $\KTj$ is generated by the subset $i_j(K)\cup\{\mathbf{T}_{j}\}$ as a ring, we have only to prove that for $x\in h_{j*}M'$, $\Delta_{j'}(s)(c\cdot x)=c\cdot (\Delta_{j'}(s)(x))$ for $c\in i_j(K)\cup\{\mathbf{T}_{j}\}$. This follows from straightforward calculation.
\end{proof}

By Lemma \ref{internal}, we obtain the family of endomorphisms $\Delta_{j'}\in \E_{\A}(\H(h_{j*}M',h_{j*}M))$ for $j'\in J,j'\neq j$.


\begin{definition}
\begin{enumerate}
\item For $M\in \CJf,j\in J,r\in (0,r(K,\partial_j)]$, we define the subset $\Fjrr M$ of $M$ by $\Fjrr M=\{x\in M;\forall s\in \H(h_{j*}M,h_{j*}f_{j,r+*}L_{j,r+}),s(x)=0,\forall s\in \H(h_{j*}M,h_{j*}f_{j,0*}L_{j,0}),s(x)\in h_{j*}f_{j,r*}L_{j,r})\}$. 
\item For $j\in J,r\in (0,r(K,\partial_j)]$, we define the endofunctor $\mathbf{F}^f_{j,[r,r]}:\Cjf\to \Cjf$ by repeating the construction of the endofunctor $F^f_{[r,r]}$ for the differential field $(K,\partial_j)$. For $M\in \Cj$, $\mathbf{F}^f_{j,[r,r]} M=\{x\in M;\forall s\in \H(M,\f_{j,r+*}\L_{r+}),s(x)=0,\forall s\in \H(M,\f_{j,0*}\L_{0}),s(x)\in \f_{j,r*}\L_{r})\}$.
\end{enumerate}
\end{definition}

\begin{lemma}\label{restriction}
\begin{enumerate}
\item For $j\in J,r\in (0,r(K,\partial_j)]$, $h_{j*}f_{j,r+*}L_{j,r+}=\f_{j,r+*}\L_{j,r+},h_{j*}f_{j,r*}L_{j,r}=\f_{j,r*}\L_{r},h_{j*}f_{j,0*}L_{j,0}=\f_{j,0*}\L_{0}$.
\item For $j\in J,r\in (0,r(K,\partial_j)]$ and $M\in \CJf$, $\Fjrr M$ is a submodule of $M$.
\item For $j\in J,r\in (0,r(K,\partial_j)]$ and $\alpha:M'\to M$ a morphism in $\CJf$, we have $\alpha(\Fjrr M')\subset \Fjrr M$. Consequently, $\alpha$ induces a morphism $\Fjrr M'\to\Fjrr M$, which is denoted by $\Fjrr \alpha$.
\item For $j\in J,r\in (0,r(K,\partial_j)]$, the function $\Fjrr$ forms an endofunctor on $\CJf$.
\end{enumerate}

\end{lemma}

\begin{proof}
1. It follows from $h_{j,r*}L_{j,r}=\L_r,h_{j,r+*}L_{j,r+}=\L_{r+},h_{j,0*}L_{j,0}=\L_0$.



2. The ring $\KTJ$ is generated by the subset $i_J(K)\cup\{T_j\}_{j\in J}$ as a ring. By part1, we have $\Fjrr M=\{x\in M;\forall s\in \H(h_{j*}M,\f_{j,r+*}\L_{j,r+}),s(x)=0,\forall s\in \H(h_{j*}M,\f_{j,0*}\L_{j,0}),s(x)\in \f_{j,r*}\L_{j,r}\}=\mathbf{F}^f_{j,[r,r]}h_{j*}M$ as subsets of $M$. 
Since $\mathbf{F}^f_{j,[r,r]}h_{j*}M$ is stable under the action of $\KTj$ on $h_{j*}M$ and $i_J(K)\cup\{T_j\}\subset h_j(\KTj)$, $\Fjrr M$ is stable under the action of $i_J(K)\cup\{T_j\}\subset \KTJ$ on $M$. 
Let $j'\in J,j'\neq j$ and $x\in \Fjrr M$. We prove $T_{j'}\cdot x\in \Fjrr M$. Let $s\in \H(h_{j*}M,h_{j*}f_{j,r+*}L_{j,r+})$. By Lemma \ref{internal}, $\Delta_{j'}(s)\in\H(h_{j*}M,h_{j*}f_{j,r+*}L_{j,r+})$. Hence $\Delta_{j'}(s)(x)=0$. By the definition of $\Delta_{j'}(s)$ and $x\in \Fjrr M$, we have $s(T_{j'}\cdot x)=0$. Let $s\in \H(h_{j*}M,h_{j*}f_{j,0*}L_{j,0})$. By an argument as above, $\Delta_{j'}(s)(x)\in h_{j*}f_{j,r*}L_{j,r}$ and $s(T_{j'}\cdot x)\in h_{j*}f_{j,r*}L_{j,r}$. Hence $T_{j'}\cdot x\in \Fjrr M$.

3. Since $\mathbf{F}^f_{j,[r,r]}h_{j*}$ is a functor, we have $\mathbf{F}^f_{j,[r,r]}h_{j*}\alpha(\mathbf{F}^f_{j,[r,r]}h_{j*}M')\subset \mathbf{F}^f_{j,[r,r]}h_{j*}M$. We obtain the assertion by using part 1 and applying the forgetful functor $\Cj\to\A$.

4. This is easily seen by part 3.
\end{proof}


\begin{lemma}\label{lift}
\begin{enumerate}
\item For $j\in J,r\in (0,r(K,\partial_j)]$, $h_{j*}\Fjrr=\mathbf{F}^f_{j,[r,r]} h_{j*}$.
\item For $j\in J,r\in (0,r(K,\partial_j)]$, the functor $\Fjrr$ is exact.
\item For $j\in J$, the obvious natural transformation $I_{j,sp}:\oplus_{r\in (0,r(K,\partial_j)]}\Fjrr\to id_{\CJf}$ is a natural isomorphism.
\item Assume $M\in \CJf$ is irreducible. Then $\supp m(h_{j*}M)$ for $j\in J$ is a one-point set, say, $\{r_j\}$ with $r_j\in (0,r(K,\partial_j)]$. Furthermore we have $\FrJs M=M$ and $F_{sp,[r'_J,r'_J]} M=0$ if $r'_J\neq r_J$.
\end{enumerate}

\end{lemma}

\begin{proof}
Part 1 follows from Lemma \ref{restriction}. Part 2 follows from the faithfulness and exactness of $h_{j*}$, the exactness of $\mathbf{F}^f_{j,[r,r]}$ and part1. To prove part 3, we may replace $\Fjrr,I_{j,sp}$ by $h_{j*}F_{j,[r,r]},h_{j*}I_{j,sp}$ respectively. By part 1, part 3 are reduced to the exactness of $\mathbf{F}^f_{j,[r,r]}$ and our previous decomposition theorem for $\Cj$. Let $M\in\Cf$ be irreducible. For $j\in J$, there exists a unique $r_j\in (0,r(K,\partial_j)]$ such that $M=\Fjrj M$ by part 3. By part 1, $\supp m(h_{j*}M)\subset \{r_j\}$. The rest of assertion is obvious.
q\end{proof}

\begin{lemma}\label{multiple exact}
\begin{enumerate}
\item  For $r_J\in (0,r(K,\partial_J)]$ and $j\in J$, we have $\FrJs \Fjrj=\FrJs$.
\item For $r_J\in (0,r(K,\partial_J)]$, the functor $\FrJs$ is exact.
\end{enumerate}
\end{lemma}

\begin{proof}
1. Let $M\in\CJf$. We have $\FrJs \Fjrj M\subset \FrJs M$ by applying $\FrJs$ to $\Fjrj M\subset M$. We have $h_{j*}\FrJs M\subset \mathbf{F}^f_{j,[r_j,r_j]} h_{j*}M=h_{j*}\Fjrj M$ by Theorems \ref{SDT} and \ref{comparisonz} and Lemma \ref{lift}. Hence $\FrJs M\subset \Fjrj M$ by applying the forgetful functor $\Cj\to Set$. We obtain $\FrJs M\subset \FrJs \Fjrj M$ by applying $\FrJs$ to both sides.

2. For $E:0\to M^{(1)}\to M^{(2)}\to M^{(3)}\to 0$ an exact sequence in $\CJf$, we prove by induction on the sum $n$ of dimension of the $M^{(i)}$'s. For a covariant endofunctor $F$ on $\CJf$ and a diagram  $D$ on $\CJf$, we denote by $FD$ the diagram obtained by applying $F$ to $D$ in an obvious way. In the base case $n=0$, we have nothing to prove. In the induction step, if there exists $j\in J$ such that $\Fjrj M^{(i)}\neq M^{(i)}$ for some $i$. We choose such a $j$. The sequence $E'=\FrJs E$ is exact by Lemma \ref{lift}. Moreover the sequence $E''=\Fjrj E'$ is exact by the induction hypothesis. By part 1, $\FrJs E$ is exact. If $\Fjrj M^{(i)}=M^{(i)}$ for all $i$ and $j\in J$ then $\FrJs M^{(i)}=M^{(i)}$ for all $i$ by definition. Hence the assertion is true in the induction step.
\end{proof}

\begin{theorem}\label{multiple SDT}
\begin{enumerate}
\item For $M\in\CJf$, we have $\FrJs M=0$ for all but finitely many $r_J\in (0,r(K,\partial_J)]$.
\item The obvious natural transformation $I_{sp}:\oplus_{r_J\in (0,r(K,\partial_J)]}\FrJs\to id_{\CJf}$ is a natural isomorphism.
\end{enumerate}

\end{theorem}
\begin{proof}
Note that any object $M$ in $\CJf$ is of finite length as $M$ is an Artinian left $\KTJ$-module. By Lemma \ref{multiple exact}, we can reduce both parts to part 4 of Lemma \ref{lift}.
\end{proof}

\begin{definition}
We define the function $m:(0,r(K,\partial_J)]\to\N$ by $m(M)(r_J)=dim \FrJs M$ for $r_J\in (0,r(K,\partial_J)]$. We also define the support of $m(M)$ by $\supp m(M):=\{r_J\in (0,r(K,\partial_J)];m(M)(r_J)\neq 0\}$.
\end{definition}

\begin{corollary}\label{multiple rationality}
\begin{enumerate}
\item The set $\supp m(M)$ is a finite set with $\#\supp m(M)\le dim M$.
\item The function $M\mapsto m(M)$ is additive on $\CJf$.
\item We have $\sum_{r_J\in (0,r(K,\partial_J)]}m(M)(r_J)=dim M$.
\item For $j\in J,r\in (0,r(K,\partial_j)]$, there exists a natural isomorphism $I_{j,r}:\oplus_{r_J:r_j=r}\FrJs\to \Fjrr$.
\item For $j\in J,r\in (0,r(K,\partial_j)]$, we have $\sum_{r_J\in (0,r(K,\partial_J)]:r_j=r}m(M)(r_J)=m(h_{j*}M)(r)$.
\item Let notation be as in \S \ref{fil}. If $m(M)(r_J)\neq 0$, then, for $j\in J$ such that $r<r(K,\partial_j)$, we have $r^{\sum_{r_J\in (0,\rKJ]:r_j=r}m(M)(r_J)}\in |K^{\times}|^{\lambda}$, where $\lambda=\Q$ when $p(K)>0$ and $\lambda=(empty)$ when $p(K)=0$.
\end{enumerate}
\end{corollary}

\begin{proof}
Parts 1,2,3 are consequences of Lemma \ref{multiple exact} and Theorem \ref{multiple SDT}. We prove part 4. We fix $j$. Let $r\in (0,r(K,\partial_j)]$ be arbitrary. As subobjects of $h_{j*}M$, we have $\oplus_{r_J:r_j=r}h_{j*}\FrJs M\subset \mathbf{F}^f_{j,[r,r]}h_{j*}M$ by Theorems \ref{SDT} and \ref{comparisonz}. Hence $\sum_{r_J:r_j=r}dim \FrJs M\le dim \Fjrr M$ and $\oplus_{r_J:r_j=r}\FrJs M\subset \Fjrr M$ as subobjects of $M$ by Lemma \ref{lift}. Since we have $\sum_{r\in (0,r(K,\partial_j)]}\sum_{r_J:r_j=r}dim \FrJs M=\sum_{r\in (0,r(K,\partial_j)]}dim \Fjrr M=dim M$ by Theorem \ref{multiple SDT} and Lemma \ref{lift}, $\oplus_{r_J:r_j=r}\FrJs M=\Fjrr M$. Part 5 is a consequence of part 4 and Lemma \ref{lift}. Part 6 follows from part 5 and Proposition \ref{rationality}.
\end{proof}

\end{document}